%% file: subgaussian.tex
\documentclass[11pt]{article}
\RequirePackage[OT1]{fontenc}
\RequirePackage{amsthm,amsmath,natbib}
\RequirePackage[colorlinks]{hyperref}
\RequirePackage{hypernat}
\usepackage{amsfonts}
\usepackage{amssymb}
\usepackage{amstext}
\usepackage{parskip}
\usepackage{fullpage}
\bibpunct{(}{)}{;}{a}{,}{,}

\usepackage{times}
\usepackage{graphicx}
\usepackage{epsfig}

\input{macros}

\def\supp{\mathop{\text{\rm supp}\kern.2ex}}
\def\conv{\mathop{\text{\rm conv}\kern.2ex}}
\def\argmin{\mathop{\text{arg\,min}\kern.2ex}}

\let\hat\widehat
\let\tilde\widetilde

\numberwithin{equation}{section}
\theoremstyle{plain}
\newtheorem{theorem}{Theorem}[section]
\newtheorem{assumption}{Assumption}[section]
\newtheorem{lemma}[theorem]{Lemma}
\newtheorem{proposition}[theorem]{Proposition}

\newtheorem{definition}[theorem]{Definition}

\newtheorem{remark}[theorem]{Remark}

{\end{eqnarray}\end{subequations}\hskip-4.0pt}%
\newenvironment{proofof}[1]{\hspace*{20pt}{\it Proof}{ of #1}.\hskip10pt}{\qed\vskip5pt}
\newenvironment{proofof2}{}{\qed\vskip5pt}
\begin{document}
\title{Restricted Eigenvalue Conditions on Subgaussian Random Matrices}

\author{Shuheng Zhou \\
  \vspace{-.3cm}\\
Seminar f\"{u}r Statistik, Department of Mathematics, ETH Z\"{u}rich, CH-8092, Switzerland
}

\date{December 20, 2009} 

\maketitle

\begin{abstract}
\noindent
It is natural to ask: what kinds of matrices satisfy the 
Restricted Eigenvalue (RE) condition? 
In this paper, we associate the RE condition (Bickel-Ritov-Tsybakov
09) with the {\em complexity} of a subset of the sphere in $\R^p$, 
where $p$ is the dimensionality of the data, and show that a class of random 
matrices with independent rows, but not necessarily independent columns, 
satisfy the RE condition, when the sample size is above a certain lower 
bound. Here we explicitly introduce an additional 
covariance structure to the class of random matrices that we have known by 
now that satisfy the Restricted Isometry Property as defined in
Cand\`es and Tao 05 (and hence the RE condition), in order to 
compose a broader class of random matrices for which the RE condition holds.
In this case, tools from geometric functional analysis in 
characterizing the intrinsic {\em low-dimensional} structures associated 
with the RE condition has been crucial in analyzing the sample complexity 
and understanding its statistical implications for high dimensional data.
\end{abstract}


\input{graphs}

\bibliography{./subgaussian}

\end{document}

%% file: macros.tex
\newcommand{\R}{\ensuremath{\mathbb R}}

\newcommand{\F}{\ensuremath{\mathcal F}}

\newcommand{\Sc}{\ensuremath{S^c}}

\newcommand{\prob}[1]{\ensuremath{{\mathbb P}\left(#1\right)}}

\newcommand{\expct}[1]{\ensuremath{{\mathbb E}#1}}

\newcommand{\size}[1]{\ensuremath{\left|#1\right|}}

\newcommand{\argmin}{\operatorname{argmin}}

\newcommand{\e}{\epsilon}
\newcommand{\ve}{\varepsilon}

\newcommand{\silent}[1]{}

\newcommand{\ip}[1]{\;\langle{\,#1\,}\rangle\;}

\newcommand{\Ball}{{B}}

\newcommand{\T}{{\mathcal T}}
\newcommand{\RE}{{\mathcal R}}

\newcommand{\X}{{\mathcal X}}

\newcommand{\Net}{{\mathcal N}}

\newcommand{\var}[1]{\ensuremath{{\textsf{Var}} \left(#1\right)}}

\newcommand{\basepen}{\ensuremath{\lambda_{\sigma,a,p}}}

\newcommand{\inv}[1]{\frac{1}{#1}}
\newcommand{\abs}[1]{\left\lvert#1\right\rvert}
\newcommand{\twonorm}[1]{\left\lVert#1\right\rVert_2}
\newcommand{\fnorm}[1]{\left\lVert#1\right\rVert_F}

\newcommand{\norm}[1]{\left\lVert#1\right\rVert}

\newcommand{\bens}{\begin{eqnarray*}}
\newcommand{\eens}{\end{eqnarray*}}
\newcommand{\ben}{\begin{eqnarray}}
\newcommand{\een}{\end{eqnarray}}

\newcommand{\beq}{\begin{equation}}
\newcommand{\eeq}{\end{equation}}

%% file: graphs.tex
\section{Introduction}

In a typical high dimensional setting, the number of variables $p$ is much 
larger than the number of observations $n$. This challenging setting appears
in linear regression, signal recovery, covariance selection in graphical 
modeling, and sparse approximations. In this paper, 
we consider recovering $\beta \in \R^p$ in the following linear model:
\beq
\label{eq::linear-model}
Y = X \beta + \epsilon,
\eeq
where $X$ is an $n \times p$ design matrix, $Y$ is a vector
of noisy observations and $\epsilon$ being the noise term. The design
matrix is treated as either fixed or random. We assume throughout this paper 
that $p \ge n$ (i.e. high-dimensional) and $\epsilon \sim N(0, \sigma^2 I_n)$. 
Throughout this paper, we assume that the columns of $X$ have $\ell_2$ 
norms in the order of $\sqrt{n}$, which holds with an overwhelming 
probability when $X$ is a random design that we shall consider.

The restricted eigenvalue (RE) conditions as 
formalized by~\cite{BRT08}~\footnote{We note the authors have defined two 
such conditions, for which we show are equivalent except on the constant
defined within each definition; see Proposition~\ref{prop:two-RE-same}
and Proposition~\ref{prop:two-RE-same-final} in 
Section~\ref{sec:two-RE} for details.} 
are among the weakest and hence the most general conditions in literature 
imposed on the Gram matrix in order to guarantee nice statistical 
properties for the Lasso and the Dantzig selector; for example, 
under this condition, they derived bounds on $\ell_2$ prediction loss and on 
$\ell_p$, where $1 \leq p \leq 2$, loss for estimating the parameters for
both the Lasso and the Dantzig selector in both linear regression and 
nonparametric regression models. From now on, we refer to their 
conditions in general as the RE condition.
Before we elaborate upon the RE condition, we need some notation and 
some more definitions to put this condition in perspective.

Consider the linear regression model in~\eqref{eq::linear-model}.
For a chosen penalization parameter $\lambda_n \geq 0$, regularized 
estimation with the $\ell_1$-norm penalty, also known as the 
Lasso \citep{Tib96} or the Basis Pursuit~\citep{Chen:Dono:Saun:1998} 
refers to the following convex optimization problem  
\begin{eqnarray}
\label{eq::origin} \; \; 
\hat \beta = \arg\min_{\beta} \frac{1}{2n}\|Y-X\beta\|_2^2 + 
\lambda_n \|\beta\|_1,
\end{eqnarray}
where the scaling factor $1/(2n)$ is chosen by convenience. 

The Dantzig selector~\citep{CT07}, for a given $\lambda_n \geq 0$, 
is defined as
\begin{eqnarray}
\label{eq::DS-func}
(DS) \; \; \arg\min_{\hat{\beta} \in \R^p} \norm{\hat \beta}_1
 \;\; \text{subject to} \;\; 
\norm{\inv{n} X^T (Y - X \hat{\beta})}_{\infty} \leq \lambda_n.
\end{eqnarray}
For an integer $1 \leq s \leq p/2$, we refer to a vector 
$\beta \in \R^{p}$ with at most $s$ non-zero entries as an $s$-sparse vector. 
Let $\beta_T \in \R^{|T|}$, be a subvector of $\beta \in \R^p$ confined to $T$.
One of the common properties of the Lasso and the Dantzig 
selector is: for an appropriately chosen $\lambda_n$, 
for a vector $\upsilon := \hat{\beta} - \beta$,  where 
$\beta$ is an $s$-sparse vector and $\hat{\beta}$ is the solution 
from either the Lasso or the Dantzig selector, 
it holds with high probability (cf. Section~\ref{sec:recovery-proof})
\ben
\label{eq::cone}
\norm{\upsilon_{I^c}}_1 \leq k_0 \norm{\upsilon_{I}}_1,
\een
where $I \subset \{1, \ldots, p\}$, $|I| \leq s$ is the support of 
$\beta$, $k_0=1$ for the Dantzig selector, and for the Lasso 
it holds for $k_0 = 3$; see~\cite{BRT08} and~\cite{CT07} in case 
columns of $X$ have $\ell_2$ norm $\sqrt{n}$.
We use $\upsilon_{T_0}$ to always represent the subvector of 
$\upsilon \in \R^p$ confined to $T_0$, which corresponds to the locations
of the $s$ largest coefficients of $\upsilon$ in absolute values:
then~\eqref{eq::cone} implies that (see Proposition~\ref{prop:magic-cone})
\ben
\label{eq::extreme}
\norm{\upsilon_{T_0^c}}_1 \leq k_0 \norm{\upsilon_{T_0}}_1.
\een
We are now ready to introduce the Restricted Eigenvalue assumption
that is formalized in~\cite{BRT08}. In Section~\ref{sec:estimators},
we show the convergence rate on $\ell_p$ for $p =1, 2$ for both the Lasso
and the Dantzig selector under this condition for the purpose of 
completeness.
\begin{assumption}
\textnormal{(\bf{Restricted Eigenvalue assumption $RE(s, k_0, X)$}
~\citep{BRT08})} 
\label{def:BRT-cond}
For some integer $1\leq s \leq p$ and a positive number $k_0$, 
the following holds:
\beq
\label{eq::admissible}
\inv{K(s, k_0, X)} \stackrel{\triangle}{=}
\min_{\stackrel{J_0 \subseteq \{1, \ldots, p\},}{|J_0| \leq s}}
\min_{\stackrel{\upsilon \not=0,}{\norm{\upsilon_{J_0^c}}_1 \leq k_0 
\norm{\upsilon_{J_0}}_1}}
\; \;  \frac{\norm{X \upsilon}_2}{\sqrt{n}\norm{\upsilon_{J_0}}_2} > 0.
\eeq
\end{assumption}

\begin{definition}
\label{def:admit}
Throughout this paper, we say that a vector $\upsilon \in \R^p$ is 
admissible to~\eqref{eq::admissible} , 
or equivalently to~\eqref{eq::admissible-random},
for a given $k_0 > 0$ as defined therein, if 
$\upsilon \not= 0$ and for some $J_0 \in \{1, \ldots, p \}$ 
such that $|J_0| \leq s$, it holds that
$\norm{\upsilon_{J_0^c}}_1 \leq k_0 \norm{\upsilon_{J_0}}_1$.
Now it is clear that if $\upsilon$ is admissible to~\eqref{eq::admissible} ,
or equivalently to~\eqref{eq::admissible-random},
~\eqref{eq::extreme} holds (cf. Proposition~\ref{prop:magic-cone}).
\end{definition}

If $RE(s, k_0, X)$ is  satisfied with $k_0 \geq 1$,
then the square submatrices of size $\leq 2s$ of $X^T X/n$ are 
necessarily positive definite (see~\cite{BRT08}).
We note the ``universality'' 
of this condition as it is not tailored to any particular set $J_0$.
We also note that given such a universality condition,
it is sufficient to check if for all $\upsilon \not = 0$ that is admissible 
to~\eqref{eq::admissible} and for $K(s, k_0, X) >0$, the following 
inequality
\ben
\label{eq::admissible-interpret}
\frac{\norm{X \upsilon}_2}{\sqrt{n}} & \geq & \frac{\norm{\upsilon_{T_0}}_2}{K(s, k_0, X)} \; > \; 0
\een
holds, where $T_0$ corresponds to locations of the $s$ largest coefficients of 
$\upsilon$ in absolute values, 
as~\eqref{eq::admissible-interpret}
is both necessary and also sufficient to guarantee that~\eqref{eq::admissible}
holds;
See Proposition~\ref{prop:magic-cone} for details.

A special class of design matrices that satisfy the RE condition
are the random design matrices. 
This is shown in a large body of work in the high dimensional setting,
for example~\citep{CRT06,CT05,CT07,BDDW08,MPT08,ALPT09},
which shows that a uniform uncertainty principle (UUP, a condition that is 
stronger than the RE condition, see~\cite{BRT08}) 
holds for ``generic'' or random design matrices for very significant 
values of $s$; roughly speaking, UUP holds when the $2s$-restricted isometry 
constant $\theta_{2s}$ is small, which we now define.
Let $X_T$, where $T \subset \{1, \ldots, p\}$ be the $n \times |T|$ 
submatrix obtained by extracting columns of $X$ indexed by $T$.
\begin{definition}\textnormal{~\citep{CT05}} 
For each integer $s  =1, 2, \ldots$, 
the $s$-restricted isometry constant $\theta_s$ of 
$X$ is the smallest quantity such that
\ben
\label{eq::RIP}
(1 - \theta_s) \twonorm{c}^2 \leq \twonorm{X_T c}^2/n 
\leq (1 + \theta_s) \twonorm{c}^2,
\een
for all $T \subset \{1,\ldots, p\}$ with $|T| \leq s$ and coefficients 
sequences $(c_j)_{j \in T}$.
\end{definition}

It is well known that for a random matrix 
the UUP holds for $s = O(n/\log(p/n))$ with i.i.d. Gaussian random 
variables (that is, Gaussian random ensemble, subject to normalizations 
of columns), the Bernoulli, and in general the subgaussian 
ensembles~\citep{BDDW08,MPT08} (cf. Theorem~\ref{thm:subgaussian}).
Recently, it is shown~\citep{ALPT09} that UUP holds for 
$s = O(n/\log^2(p/n))$ when $X$ is a random matrix composed of 
columns that are independent isotropic vectors with log-concave densities.
Hence this setup only requires $\Theta(\log(p/n))$ 
or $\Theta(\log^2(p/n))$ observations per
nonzero value in $\beta$, where $\Theta$ hides a very small constant,
when $n$ is a nonnegligible fraction of $p$, in order to perform accurate 
statistical estimation; we call this level of sparsity as the linear sparsity.

The main purpose of this paper is to extend the family of random matrices 
from the i.i.d. subgaussian ensemble $\Psi$ (cf.~\eqref{eq::gamma-func}), 
which are now well known to satisfy the UUP condition and hence the RE 
condition under linear sparsity, to a larger family of random matrices 
$X := \Psi \Sigma^{1/2}$,
where $\Sigma$ is assumed to behave {\em sufficiently nicely} in the sense 
that it satisfies certain restricted eigenvalue conditions to be defined in 
Section~\ref{sec:random-RE}. Thus we have explicitly introduced the 
additional covariance structure $\Sigma$ to the columns of $\Psi$
in generating $X$.
In Theorem~\ref{thm:subgaussian-T}, we show that $X$ satisfies the RE 
condition with overwhelming probability once we have 
$n \geq C s \log (c p/s)$,  where $c$ is an absolute constant and $C$ 
depends on the restricted 
eigenvalues of $\Sigma$ (cf.~\eqref{eq::C-define}), when $\Sigma$ satisfies 
the restricted eigenvalue assumption to be specified in 
Section~\ref{sec:random-RE}. We believe such results can be extended 
to other cases: for example, when $X$ is the composition of a random 
Fourier ensemble, or randomly sampled rows of orthonormal matrices, see for 
example~\cite{CT06,CT07}.

Finally, we show rate of convergence results for the Lasso and the Dantzig
selector given such random matrices. Although such results are almost entirely
known, we provide a complete  analysis for a self-contained presentation.
Given these rates of convergence (cf. Theorem~\ref{thm:lasso} and
Theorem~\ref{thm:DS}),  one can exploit thresholding algorithms to adjust
the bias and get rid of excessive variables selected by an initial estimator 
relying on $\ell_1$ regularized minimization functions, for example,  
the Lasso or the Dantzig selector;  under the UUP or the RE type of conditions,
such procedures are shown to  select a sparse model, which contains the set 
of  variables in $\beta$ that are significant in their  absolute values; 
in addition, one can then conduct an ordinary least squares regression on 
such a sparse model to obtain a final estimator, whose bias is significantly 
reduced compared to the initial estimators. Such algorithms are proposed 
and analyzed in a series of papers, 
for example~\cite{CT07,MY09,WR08,Zhou09}.

\subsection{Restricted eigenvalue assumption for a random design}
\label{sec:random-RE}
We will define the family of random matrices that we consider
and the restricted eigenvalue assumption that we impose on such a 
random design. We need some more definitions.
\begin{definition}
\label{def:psi2-vector}
Let $Y$ be a random vector in $\R^p$; $Y$ is called isotropic
if for every $y \in \R^p$, $\expct{\abs{\ip{Y, y}}^2} = \twonorm{y}^2$,
and is $\psi_2$ with a constant $\alpha$ if for every $y \in \R^p$,
\beq
\norm{\ip{Y, y}}_{\psi_2} := \;
\inf \{t: \expct{\exp(\ip{Y,y}^2/t^2)} \leq 2 \} 
\; \leq \; \alpha \twonorm{y}.
\eeq
\end{definition}
The important examples of isotropic, subgaussian vectors are the Gaussian
random vector $Y = (h_1, \ldots, h_p)$ where $h_i, \forall i$ 
are independent $N(0, 1)$ random variables, and the random vector 
$Y = (\ve_1, \ldots, \ve_p)$ where $\ve_i, \forall i$ are independent, 
symmetric $\pm 1$ Bernoulli random variables. 

A subgaussian or $\psi_2$ operator is a random operator 
$\Gamma: \R^p \to \R^n$ of the form 
\ben
\label{eq::gamma-func}
\Gamma = \sum_{i=1}^n \ip{\Psi_i,\cdot} e_i,
\een
where $e_1, \ldots, e_n$ are the canonical basis of $\R^n$ and
$\Psi_1, \ldots, \Psi_n$ are independent copies of an isotropic $\psi_2$ 
vector $\Psi_0$ on $\R^p$. 
Note that throughout this paper, $\Gamma$ is represented by a 
random matrix $\Psi$ whose rows are $\Psi_1, \ldots, \Psi_n$.
Throughout this paper, we consider a random design matrix $X$ that is 
generated as follows:
\ben
\label{eq::rand-des}
X := \Psi \Sigma^{1/2}, \; \; \mbox{ where we assume } 
\Sigma_{jj} = 1, \forall j = 1, \ldots, p,
\een
and $\Psi$ is a random matrix whose rows $\Psi_1, \ldots, \Psi_n$ 
are independent copies of an isotropic $\psi_2$ vector 
$\Psi_0$ on $\R^p$ as in Definition~\ref{def:psi2-vector}. 
For a random design $X$ as in~\eqref{eq::rand-des}, we make the following
assumption on $\Sigma$. A slightly stronger condition has been originally 
defined in~\cite{ZGB09} in the context of Gaussian graphical modeling.

\begin{assumption}
\textnormal{\bf Restricted eigenvalue condition $RE(s, k_0, \Sigma)$.}
\label{def:memory}
Suppose $\Sigma_{jj} = 1, \forall j = 1, \ldots, p$, and for some integer 
$1\leq s \leq p$ and a positive number $k_0$, the following condition holds,
\beq
\label{eq::admissible-random}
\inv{K(s, k_0, \Sigma)} := \min_{\stackrel{J_0 \subseteq \{1, \ldots,
    p\},}{|J_0| \leq s}} 
\min_{\stackrel{\upsilon \not=0,}
{\norm{\upsilon_{J_0^c}}_1 \leq k_0 \norm{\upsilon_{J_0}}_1}}
\; \;  \frac{\norm{\Sigma^{1/2} \upsilon}_2}{\norm{\upsilon_{J_{0}}}_2} > 0.
\eeq
\end{assumption}
We note that similar to the case in Assumption~\ref{def:BRT-cond},
it is sufficient to check if for  $\upsilon \not = 0$ that is admissible 
to~\eqref{eq::admissible-random} and for $K(s, k_0, \Sigma) >0$, that
the following inequality
\ben
\label{eq::admissible-interpret-random}
\twonorm{\Sigma^{1/2} \upsilon} & \geq & 
\frac{\norm{\upsilon_{T_0}}_2}{K(s, k_0, \Sigma)} \; > \; 0
\een
holds, where $T_0$ corresponds to locations of the $s$ largest coefficients of 
$\upsilon$ in absolute values. Formally, we have
\begin{proposition}
\label{prop:magic-cone}
Let $1 \leq s \leq p/2$ be an integer and $k_0 > 0$.
Suppose $\delta \not= 0$ is admissible to~\eqref{eq::admissible-random},
or equivalently to~\eqref{eq::admissible} , in the sense of 
Definition~\ref{def:admit}; then
\ben
\label{eq::cone-max}
\norm{\delta_{T_0^c}}_1 & \leq & k_0 \norm{\delta_{T_0}}_1; 
\een
Hence~\eqref{eq::admissible-interpret-random}  is both necessary and 
sufficient to guarantee that~\eqref{eq::admissible-random} holds.
Similarly~\eqref{eq::admissible-interpret} is a necessary and sufficient 
condition for~\eqref{eq::admissible} to hold.
Moreover, suppose that $\Sigma$ satisfies Assumption~\ref{def:memory}, 
then for $\delta$ that is admissible to~\eqref{eq::admissible-random},
we have
$$\twonorm{\Sigma^{1/2} \delta} \geq 
\frac{\twonorm{\delta_{J_0}}}{K(s, k_0, \Sigma)}  >  0.$$
\end{proposition}

We now define
\ben
\sqrt{\rho_{\min}(m)} & := & \min_{\stackrel{\twonorm{t} = 1}
{|\supp(t)| \leq m}} \; \; \twonorm{\Sigma^{1/2} t}, \\
\label{eq::eigen-Sigma}
\sqrt{\rho_{\max}(m)} & := & \max_{\stackrel{\twonorm{t} = 1}
{|\supp(t)| \leq m}} \; \; \twonorm{\Sigma^{1/2} t},
\een
where we assume that $\sqrt{\rho_{\max}(m)}$ is a constant for $m \leq p/2$.
If $RE(s, k_0, \Sigma)$ is  satisfied with $k_0 \geq 1$,
then the square submatrices of size $\leq 2s$ of $\Sigma$ are 
necessarily positive definite (see~\cite{BRT08}); hence
throughout this paper, we also assume that
\beq
\label{eq::eigen-admissible-random}
\rho_{\min}(2s) > 0.
\eeq
Note that when $\Psi$ is a Gaussian random matrix with i.i.d. $N(0, 1)$ 
random variables, $X$ as in~\eqref{eq::rand-des} corresponds to a random 
matrix with independent rows, such that each row is a random vector that 
follows a multivariate normal distribution $N(0, \Sigma)$: 
\begin{eqnarray}
\label{eq::rand-des-gauss}
X \ \mbox{has i.i.d. rows}\ \sim N(0, \Sigma),
\mbox{ where we assume } 
\Sigma_{jj} = 1, \forall j = 1, \ldots, p.
\end{eqnarray}
Finally, we need the following notation. For a set $V \subset \R^p$,
we let $\conv V$ denote the convex hull of $V$. For a finite set
$Y$, the cardinality is denoted by $|Y|$. Let $\Ball_2^p$ and 
$S^{p-1}$ be the unit Euclidean ball and the unit sphere respectively.

\subsection{The main theorem}
Throughout this section, we assume that $\Sigma$ 
satisfies~\eqref{eq::admissible-random} and~\eqref{eq::eigen-Sigma}
for $m = s$.  We assume $k_0 > 0$ and it is understood to be the 
same quantity throughout our discussion. Let us define
\beq
\label{eq::C-define}
\bar{C} = 3 (2 + k_0) K(s, k_0, \Sigma) \sqrt{\rho_{\max}(s)},
\eeq
where $k_0>0$ is understood to be the same as in~\eqref{eq::cone-constraint}.
Our main result in Theorem~\ref{thm:subgaussian-T} roughly says that for 
a random matrix $X := \Psi \Sigma^{1/2}$, which is 
the product of a random subgaussian ensemble $\Psi$ and a fixed positive 
semi-definite matrix $\Sigma^{1/2}$, the RE condition will be satisfied with 
overwhelming probability, given $n$ that is sufficiently large
(cf.~\eqref{eq::sample-size-gen}).
Before introducing the theorem formally, we define the class of vectors $E_s$, 
for a particular integer $1 \leq s \leq p/2$, that are relevant to the 
RE Assumption~\ref{def:BRT-cond} and~\ref{def:memory}.
For any given subset $J_0 \subset \{1, \ldots, p\}$ such that $|J_0| \leq s$,
we consider the set of vectors $\delta$ such that
\ben
\label{eq::cone-constraint}
\norm{\delta_{J_0^c}}_1 \leq k_0 \norm{\delta_{J_0}}_1
\een
holds for some $k_0 > 0$, subject to a normalization condition such that 
$\Sigma^{1/2} \delta \in S^{p-1}$; we then define the set $E_s'$ 
as unions of all vectors that satisfy the cone constraint 
as in~\eqref{eq::cone-constraint} with respect to any index set 
$J_0 \subset \{1, \ldots, p\}$ such that $|J_0| \leq s$;
\begin{eqnarray*}
\label{eq::elip}
E_s' = \left\{\delta: \twonorm{ \Sigma^{1/2} \delta} = 1
\; s.t. \; \exists J_0 \subseteq \{1, \ldots, p\} \; s.t. \; |J_0| \leq s
\text{ and~\eqref{eq::cone-constraint} holds} \right\}.
\end{eqnarray*}
We now define a even broader set: let $\delta_{T_0}$ be the 
subvector of $\delta$ confined to the locations of its $s$ largest 
coefficients:
\begin{eqnarray*}
\label{eq::elip}
E_s = \left\{\delta: \twonorm{\Sigma^{1/2} \delta} = 1
\; s.t. \; \norm{\delta_{T_0^c}}_1 \leq k_0 \norm{\delta_{T_0}}_1
\text{ holds,} \right\}
\end{eqnarray*}
\begin{remark}
\label{rem:twosets}
It is clear from Proposition~\ref{prop:magic-cone} that $E_s' \subset E_s$
for the same $k_0 > 0$.
\end{remark}
Theorem~\ref{thm:subgaussian-T} is
the main contribution of this paper.
\silent{
In more detail, for a given sparsity $1 \leq s \leq p/2$, we consider 
all subgaussian random matrices $\Psi_{n \times p}$ with $n$ 
independent isotropic $\psi_2$ random vectors in $\R^p$ 
(where elements in each row need not be independent) being its row vectors 
such that the RIP condition as in~\eqref{eq::RIP} holds with overwhelming 
probability when the sample size $n = \Theta(s \log (p/s))$ is large enough 
with respect to $p$ and $s$; }

\begin{theorem}
\label{thm:subgaussian-T}
Set $1 \leq n \leq p$, $0< \theta < 1$, and $s \leq p/2$. 
Let $\Psi_0$ be an isotropic 
$\psi_2$ random vector on $\R^p$ with constant $\alpha$ as in
Definition~\ref{def:psi2-vector} and $\Psi_1, \ldots, \Psi_n$ be 
independent copies of $\Psi_0$. 
Let $\Psi$ be a random matrix in $\R^{n \times p}$ whose rows are 
$\Psi_1, \ldots, \Psi_n$.
Let $\Sigma$  satisfy~\eqref{eq::admissible-random} 
and~\eqref{eq::eigen-Sigma}.
If $n$ satisfies for $\bar{C}$ as defined in~\eqref{eq::C-define}
\ben
\label{eq::sample-size-gen}
n >  \frac{c' \alpha^4}{\theta^2} 
\max\left(\bar{C}^2 s \log (5ep/s), 9 \log p\right),
\een
then with probability at least $1- 2\exp(-\bar{c} \theta^2 n/\alpha^4)$, 
we have for all $\delta \in E_s$,
\begin{eqnarray}
\label{eq::phi-bound}
1 - \theta & \leq &
\frac{\twonorm{\Psi \Sigma^{1/2} \delta}}{\sqrt{n}}\; \leq \; 1 + \theta, 
\; \; \text{ and } \\
\label{eq::phi-bound-column}
 \forall \rho_i, \; \; \; 
1 - \theta & \leq & \frac{\twonorm{\Psi \rho_i}}{\sqrt{n}} 
\; \leq \; 1 + \theta,
\end{eqnarray}
where $\rho_1, \ldots, \rho_p$ are column vectors of $\Sigma^{1/2}$, and 
$c', \bar{c} >0$ are absolute constants.
\end{theorem}
We now state some immediate consequences of 
Theorem~\ref{thm:subgaussian-T}.
Consider the random design $X = \Psi \Sigma^{1/2}$ as defined
in Theorem~\ref{thm:subgaussian-T}. It is clear when all columns of 
$X$ have an Euclidean norm close to $\sqrt{n}$,
as guaranteed by~\eqref{eq::phi-bound-column} for 
$0< \theta<1$ that is small, it makes sense to discuss the RE 
condition in the form of~\eqref{eq::admissible}. 
We now define the following event $\RE$ on a random design $X$,
which provides an upper bound on $K(s, k_0, X)$ for a given $k_0 >0$, 
when $X$ satisfies Assumption $RE(s, k_0, X)$:
\begin{eqnarray}
\label{eq::good-random-design-RE}
\RE(\theta) := \left\{X: RE(s, k_0, X) \; \text{ holds with } \; 
0 < K(s, k_0, X) \leq \frac{K(s, k_0, \Sigma)}{1-\theta} \right\}
\end{eqnarray}
Under Assumption~\ref{def:memory}, we consider the set of vectors 
$u := \Sigma^{1/2} \delta$, 
where $\delta \not= 0$ is admissible to~\eqref{eq::admissible-random},
and show a uniform bound on the concentration of each individual random 
variable of the form $\twonorm{\Gamma u}^2 
:= \twonorm{X \delta}^2$ around its mean.
By  Proposition~\ref{prop:magic-cone}, we have 
$\twonorm{u} = \twonorm{\Sigma^{1/2} \delta} > 0$.
We can now apply \eqref{eq::phi-bound} to each 
$(\delta /\twonorm{\Sigma^{1/2} \delta}) \not=0$,
which belongs to $E_s'$ and hence $E_s$ (see Remark~\ref{rem:twosets}),
and conclude that 
\begin{eqnarray}
\label{eq::X-bound}
0< (1 - \theta) \twonorm{\Sigma^{1/2} \delta} \; \leq \; 
\frac{\twonorm{X \delta}}{\sqrt{n}} 
& \leq & (1 + \theta) \twonorm{\Sigma^{1/2} \delta}
\end{eqnarray}
hold for all $\delta \not= 0$ that is 
admissible to~\eqref{eq::admissible-random},
with probability at least $1- 2 \exp(-\bar{c} \theta^2 n/\alpha^4)$.
Now the lower bound in~\eqref{eq::X-bound} implies that
\ben
\label{eq::admissible-instance}
\frac{\twonorm{X \delta}}{\sqrt{n}} \geq
(1 - \theta) \twonorm{\Sigma^{1/2} \delta} \; \geq \;
(1 - \theta) \frac{ \twonorm{\delta_{T_0}} }{K(s, k_0, \Sigma)} > 0,
\een
where $T_0$ is the locations of largest  coefficients of $t$ in absolute values.
Hence~\eqref{eq::phi-bound-column} and
event $\RE(\theta)$ hold simultaneously,
with probability at least $1- 2 \exp(-\bar{c} \theta^2 n/\alpha^4)$,
given~\eqref{eq::admissible-interpret-random}
and Proposition~\ref{prop:magic-cone},
so long as $n$ satisfies~\eqref{eq::sample-size-gen}.
\begin{remark}
It is clear that this result generalizes the notion of restricted 
isometry property (RIP) introduced in~\cite{CT05}. In particular, when
$\Sigma = I$ and $\delta$ is $s$-sparse,~\eqref{eq::RIP} holds for $X$
with $\theta_s = \theta$, given~\eqref{eq::X-bound}.
\end{remark}

\section{Proof Theorem~\ref{thm:subgaussian-T}}
\label{sec:sub-gauss}
In this section, we first state a definition and then two lemmas
in Section~\ref{sec:complexity}, from which we show the proof of
Theorem~\ref{thm:subgaussian-T} in Section~\ref{sec:main-proof}.
We shall identify the basis with the canonical basis 
$\{e_1, e_2, \ldots, e_p\}$ of $\R^p$, where 
$e_i = \{0, \ldots, 0, 1, 0, \ldots, 0\}$, and
it is to be understood that $1$ appears in the $i$th position and
$0$ appears elsewhere.
\begin{definition}
For a subset $V \subset \R^p$, we let
\beq
\label{eq::entropy}
\ell_*(V) = \expct{\sup_{t \in V} \abs{\sum_{i=1}^p g_i t_i}}
\eeq
where $t = (t_i)^p_{i=1} \in \R^p$ and $g_1, \ldots, g_p$ are independent
$N(0, 1)$ Gaussian random variables.
\end{definition}

\subsection{The complexity measures}
\label{sec:complexity}
The subset  $\Upsilon$ that is relevant to our result is a subset
of the sphere $S^{p-1}$ such that the linear function 
$\Sigma^{1/2}: E_s \to \R^p$ maps $\delta \in E_s$ onto:
\ben
\label{eq::sparse-sphere}
\Upsilon & := & \Sigma^{1/2}(E_s) = \{v \in \R^p:  v = \Sigma^{1/2} \delta 
\text{ for some } \delta \in E_s\}.
\een

\silent{
It is clear that the following assertion holds:
For two subsets $T_1, T_2 \subset \R^p$, 
\bens
\label{eq::entropy-union}
\ell_*(T_1 \cup T_2) 
& = & \expct{\sup_{t \in T_1 \cup T_2} \abs{\sum_{i=1}^p g_i t_i}} \\
& \leq & 
\expct{\sup_{t \in T_1} \abs{\sum_{i=1}^p g_i t_i}} +
\expct{\sup_{t \in T_2} \abs{\sum_{i=1}^p g_i t_i}}  \\
& =: & 
\ell_*(T_1) + \ell_*(T_2)
\eens
where $t = (t_i)^p_{i=1} \in \R^p$ and $g_1, \ldots, g_p$ are independent
$N(0, 1)$ Gaussian random variables.
}

We now show a bound on functional of $\ell_*(\Upsilon)$,
for which we crucially exploit the cone property of vectors in $E_s$, the RE 
condition on $\Sigma$, and the bound of $\rho_{\max}(s)$. 
Lemma~\ref{lemma:U-number} is one of the main technical
contributions of this paper.
\begin{lemma}\textnormal{(Complexity of a subset of $S^{p-1}$)}
\label{lemma:U-number}
Let $\Sigma$ satisfy~\eqref{eq::admissible-random} and~\eqref{eq::eigen-Sigma}.
Let $h_1, \ldots, h_p$ be independent $N(0, 1)$ random variables.
Let $1 \leq s \leq p/2$ be an integer. Then
\ben 
 \label{eq::Sigma-delta-entropy} 
\ell_*(\Upsilon) & := & \expct{\sup_{y \in \Upsilon}} 
\abs{\sum_{i=1}^p h_i y_i} 
= \expct{\sup_{\delta \in E_s} \abs{\ip{h, \Sigma^{1/2} \delta}} }
\leq \bar{C} \sqrt{s \log (cp/s)}
\een
where $\bar{C}$ is defined in~\eqref{eq::C-define} and $c = 5e$. 
\end{lemma}

\begin{remark}
We will also show in our fundamental proof for the zero-mean Gaussian 
random ensemble with covariance matrix being $\Sigma$,
where such complexity measure is used exactly in Section~\ref{sec:Gaussian}.
There we also give explicit constants.
\end{remark}

Now let $\Sigma^{1/2} := (\rho_{ij})$ and 
$\rho_1, \ldots, \rho_p$ denote its $p$ column vectors.
By definition of $\Sigma = (\Sigma^{1/2})^2$, it holds that 
$\twonorm{\rho_i}^2 = \sum_{j = 1}^p \rho_{ij}^2 = \Sigma_{ii} = 1,$
for all $i = 1, \ldots, p$. Thus we have the following.
\begin{lemma}
\label{lemma:column-space}
Let $\Phi =\{\rho_1, \ldots, \rho_p \}$ be the subset of vectors
in $S^{p-1}$ that correspond to columns of $\Sigma^{1/2}$. 
It holds that $\ell_*(\Phi) \leq 3 \sqrt{\log p}.$
\end{lemma}

\subsection{Proof of Theorem~\ref{thm:subgaussian-T}}
\label{sec:main-proof}
\begin{proofof2}
The key idea to prove Theorem~\ref{thm:subgaussian-T} is 
to apply the powerful Theorem~\ref{thm:subgaussian} as shown 
in~\cite{MPT07,MPT08}(Corollary 2.7, Theorem 2.1 respectively)
to the subset $\Upsilon$ of the sphere $S^{p-1}$, 
as defined in~\eqref{eq::sparse-sphere}.
As explained in~\cite{MPT08}, in the context of Theorem~\ref{thm:subgaussian},
the functional $\ell_*(\Upsilon)$ is the complexity measure of the 
set $\Upsilon$, which measures the extent in which probabilistic bounds 
on the concentration of each individual random variable of the form 
$\twonorm{\Gamma v}^2$ around its mean can be combined to form a bound 
that holds uniformly for all $v \in \Upsilon$. 
\begin{theorem}\textnormal{~\citep{MPT07,MPT08}}
\label{thm:subgaussian}
Set $1 \leq n \leq p$ and $0< \theta < 1$. Let $\Psi$ be an isotropic 
$\psi_2$ random vector on $\R^p$ with constant $\alpha$, and 
$\Psi_1, \ldots, \Psi_n$
be independent copies of $\Psi$.
Let $\Gamma$ be as defined in~\eqref{eq::gamma-func} and let 
$V \subset S^{p-1}$. If $n$ satisfies
\ben
\label{eq::sample-size-MPT}
n > \frac{c' \alpha^4}{\theta^2} \ell_*(V)^2,
\een
Then with probability at least $1- \exp(-\bar{c} \theta^2 n/\alpha^4)$, 
for all $v \in V$, we have
\begin{eqnarray}
\label{eq::Gamma-RIP}
1 - \theta \; \leq \; \twonorm{\Gamma v}/\sqrt{n} & \leq & 1 + \theta,
\end{eqnarray}
where $c', \bar{c} >0$ are absolute constants.
\end{theorem}
It is clear that~\eqref{eq::phi-bound} follows immediately from 
Theorem~\ref{thm:subgaussian} by having $V = \Upsilon$, 
given Lemma~\ref{lemma:U-number}.
In fact, we can now finish proving Theorem~\ref{thm:subgaussian-T}
by applying Theorem~\ref{thm:subgaussian} twice, 
by having $V = \Upsilon$ and $V = \Phi$ respectively:
the lower bound on $n$ is obtained by applying the upper bounds
on $\ell_*(\Upsilon)$ as given in Lemma~\ref{lemma:U-number} and
on $\ell_*(\Phi)$ as in Lemma~\ref{lemma:column-space}.
We then apply the union bound to bound the probability of the bad
events when~\eqref{eq::Gamma-RIP} does not hold for some
$v \in \Upsilon$ or some $v \in \Phi$ respectively.
\end{proofof2}

\section{$\ell_p$ convergence for the Lasso and the Dantzig selector}
\label{sec:estimators}
Throughout this section, we assume that  $0< \theta < 1$,
and $c', \bar{c} >0$ are absolute constants.
Conditioned on the random design as in~\eqref{eq::rand-des} 
satisfying properties as guaranteed in Theorem~\ref{thm:subgaussian-T}, 
we proceed to treat $X$ as a deterministic design, for which
both the RE condition as described in~\eqref{eq::good-random-design-RE} 
and condition $\F(\theta)$ defined as below hold, 
\begin{eqnarray}
\label{eq::good-random-design-diag}
\F(\theta) := \left\{X: \forall j = 1, \ldots, p,\;
1 - \theta \leq \frac{\twonorm{X_j}}{\sqrt{n} } \leq 1 + \theta \right\},
\end{eqnarray}
where $X_1, \ldots, X_p$ are the column vectors of $X$:
Formally, we consider the set $\X \ni X$ of random designs that 
satisfy both condition $\RE(\theta)$ and $\F(\theta)$, for some
$0< \theta < 1$.
By Theorem~\ref{thm:subgaussian-T}, we have for 
$n$ satisfy the lower bound in~\eqref{eq::sample-size-gen},
$$\prob{\X} :=  \prob{\RE(\theta) \cap \F(\theta)} 
\geq 1- 2\exp(-\bar{c} \theta^2 n/\alpha^4).$$
It is clear that on $\X$, Assumption \ref{def:memory} holds for $\Sigma$.
We now bound the correlation between the noise and covariates of $X$
for $X \in \X$, where we also define a constant $\basepen$ which is used 
throughout the rest of this paper. 
For each $a \geq 0$,  for $X \in \F(\theta)$, let 
\ben
\label{eq::low-noise}
{\T_a} := 
\biggl \{\e: \norm{\frac{X^T \e}{n}}_{\infty} \leq (1+ \theta) \basepen,
\; \text{ where } \;  X \in \F(\theta), \text{ for } 0< \theta < 1\biggr \},
\een
where $\basepen = \sigma \sqrt{1 + a} \sqrt{(2\log p)/n}$, where
$a \geq 0$; we have (cf. Proposition~\ref{lemma:gaussian-noise})
\ben
\label{eq::prob-T}
\; \; \; \;
\prob{\T_a}  \geq 
1 - (\sqrt{\pi \log p} p^a)^{-1};
\een
In fact, for such a bound to hold, we only need 
$\frac{\twonorm{X_j}}{\sqrt{n} } \leq 1 + \theta, \forall j$ 
to hold in $\F(\theta)$.
We note that constants in the theorems are not optimized.

\begin{theorem}{\textnormal{(\bf{Estimation for the Lasso})}}
\label{thm:lasso}
Set $1 \leq n \leq p$, $0< \theta < 1$, and $a > 0$.
Let $s < p/2$.
Consider the linear model in (\ref{eq::linear-model}) with random design
$X := \Psi \Sigma^{1/2}$, where $\Psi_{n \times p}$ is a subgaussian 
random matrix as defined in Theorem~\ref{thm:subgaussian-T}. 
and $\Sigma$ satisfies~\eqref{eq::admissible-random} 
and~\eqref{eq::eigen-Sigma}.
Let $\hat\beta$ be an optimal solution to the Lasso as 
in~\eqref{eq::origin} with $\lambda_n \geq 2 (1 + \theta) \basepen$.
Suppose that  $n$ satisfies for $\bar{C}$ as in~\eqref{eq::C-define},
\begin{eqnarray}
\label{eq::sparsity-condition-random}
n >  \frac{c' \alpha^4}{\theta^2} 
\max\left(\bar{C}^2 s \log (5ep/s), 9 \log p\right).
\end{eqnarray}
Then with probability at least 
$\prob{\X \cap \T_a} 
\geq 1- 2\exp(-\bar{c} \theta^2 n/\alpha^4) - \prob{\T_a^c}$,
we have for  $B \leq 4 K^2(s, 3, \Sigma)/(1-\theta)^2$ and $k_0 = 3$,
\ben
\label{eq::2-loss}
\twonorm{\hat{\beta} -\beta} \leq 2 B  \lambda_n \sqrt{s}, \; \;
\text{ and } \; \norm{\hat{\beta} -\beta}_1 \leq B \lambda_n s.
\een
\end{theorem}

\begin{theorem}{\textnormal{(\bf{Estimation for the Dantzig selector})}}
\label{thm:DS}
Set $1 \leq n \leq p$, $0< \theta < 1$, and $a > 0$. Let $s < p/2$.
Consider the linear model in (\ref{eq::linear-model}) with random design
$X := \Psi \Sigma^{1/2}$, where $\Psi_{n \times p}$ is a subgaussian 
random matrix as defined in Theorem~\ref{thm:subgaussian-T}. 
and $\Sigma$ satisfies~\eqref{eq::admissible-random} 
and~\eqref{eq::eigen-Sigma}.
Let $\hat\beta$ be an optimal solution to
the Dantzig selector as in~\eqref{eq::DS-func} where
$\lambda_n \geq (1+ \theta) \basepen$.
Suppose that  $n$ satisfies for $\bar{C}$ as in~\eqref{eq::C-define},
\begin{eqnarray}
\label{eq::sparsity-condition-random}
n >  \frac{c' \alpha^4}{\theta^2} 
\max\left(\bar{C}^2 s \log (5ep/s), 9 \log p\right).
\end{eqnarray}
then with probability at least 
$\prob{\X \cap \T_a} 
\geq 1- 2\exp(-\bar{c} \theta^2 n/\alpha^4) - \prob{\T_a^c}$,
we have for $B \leq 4 K^2(s, 1, \Sigma)/(1-\theta)^2$ and 
 $k_0 = 1$,
\ben
\label{eq::2-loss}
\twonorm{\hat{\beta} -\beta} \leq 3 B \lambda_n \sqrt{s},
\; \text{and } \; \norm{\hat{\beta} -\beta}_1 \leq 2 B \lambda_n s.
\een
\end{theorem}

Proofs are given in Section~\ref{sec:recovery-proof}. 

\noindent{\bf Acknowledgments.}
Research is supported by the Swiss National Science Foundation (SNF)
Grant 20PA21-120050/1.
The author is extremely grateful to 
 Guillaume Lecu\'{e} for his careful reading of the manuscript and 
for his many constructive and insightful comments that  have 
lead to a significant improvement of the presentation of this paper.
The author would also like to thank Olivier Gu\'{e}don, 
Alain Pajor, and Larry Wasserman for helpful conversations,
and Roman Vershynin for providing a reference.
\appendix

\section{Some preliminary propositions}
\label{sec:append-pre}
In this section, we first prove Proposition~\ref{prop:magic-cone}, 
in Section~\ref{sec:magic-cone},  which is used throughout the 
rest of the paper.
We then present a simple decomposition for vectors $\delta \in E_s$ and 
show some immediate implications, which we shall need in the proofs for 
Lemma~\ref{lemma:U-number}, Theorem~\ref{thm:lasso} and 
 Theorem~\ref{thm:DS}.

\subsection{Proof of Proposition~\ref{prop:magic-cone}}
\label{sec:magic-cone}
\begin{proofof2}
For each $\delta$ that is admissible to~\eqref{eq::admissible-random},
there exists a subset of indices
$J_0 \subseteq \{1, \ldots, p\}$ such that both $|J_0| \leq s$ and
$\norm{\delta _{J_0^c}}_1 \leq k_0 \norm{\delta_{J_0}}_1$ hold.
This immediately implies that~\eqref{eq::cone-max} holds for $k_0 > 0$,
\bens
\norm{\delta_{T_0^c}}_1 = \norm{\delta}_1 -  \norm{\delta_{T_0}}_1 
\leq \norm{\delta}_1 -  \norm{\delta_{J_0}}_1 = \norm{\delta_{J_0^c}}_1  
\leq k_0 \norm{\delta_{J_0}}_1 \leq k_0 \norm{\delta_{T_0}}_1
\eens
due to the maximality of $\norm{\delta_{T_0}}_1$ among all 
$\norm{\delta_{J_0}}_1$ for $J_0 \subseteq \{1, \ldots, p\}$ 
such that $|J_0| \leq s$.  This immediately implies that  
$E_s' \subset E_s$.

We now show that
~\eqref{eq::admissible-interpret-random} is a necessary and sufficient 
condition for~\eqref{eq::admissible-random} to hold;
the same argument applies to the RE conditions on $X$.
Suppose~\eqref{eq::admissible-interpret-random} hold
for $\delta \not= 0$; we have for all $J_0 \in \{1, \ldots, p \}$ 
such that $|J_0| \leq s$ 
and $\norm{\delta_{J_0^c}}_1 \leq k_0 \norm{\delta_{J_0}}_1$,
\ben
\label{eq::moreover}
\norm{\Sigma^{1/2} \delta}_2
\geq 
\frac{\norm{\delta_{T_0}}_2}{K(s, k_0, \Sigma)}
\geq 
\frac{\norm{\delta_{J_0}}_2}{K(s, k_0, \Sigma)} > 0,
\een
where the last inequality is due to the fact that
$\twonorm{\delta_{J_0}} > 0$; Suppose $\twonorm{\delta_{J_0}} = 0$ 
otherwise; then $\norm{\delta_{J_0^c}}_1 \leq k_0 \norm{\delta_{J_0}}_1 = 0$ 
would imply that $\delta = 0$, which is a contradiction.
Conversely, suppose that~\eqref{eq::admissible-random} hold;
then~\eqref{eq::admissible-interpret-random} must also hold, given that 
$T_0$ satisfies~\eqref{eq::cone-max} with $|T_0| = s$,
and $\delta_{T_0}  \not= 0$.

Finally, the ``moreover'' part holds given Assumption~\ref{def:memory},
in view of~\eqref{eq::moreover}.
\end{proofof2}

\subsection{Decomposing a vector in $E_s$}
\label{sec:decomp-vec}
For each $\delta \in E_s$, we decompose $\delta$ into a set of vectors
$\delta_{T_0}$, $\delta_{T_1}$,  $\delta_{T_2}$, \ldots, $\delta_{T_K}$ 
such that $T_0$ corresponds to locations of the $s$ largest 
coefficients of $\delta$ in absolute values,
$T_1$ corresponds to locations of the $s$ largest coefficients of 
$\delta_{T_0^c}$ in absolute values, $T_2$ corresponds to locations of the 
next $s$ largest coefficients of $\delta_{T_0^c}$ in absolute values, and so on. 
Hence we have  $T_0^c = \bigcup_{k=1}^K T_k$, where $K \geq 1, 
|T_k| = s, \forall k = 1, \ldots, K-1$, and $|T_K| \leq s$. 
Now for each $j \geq 1$, we have 
$$\twonorm{\delta_{T_j}} \leq \sqrt{s} \norm{\delta_{T_j}}_{\infty} \leq 
\inv{\sqrt{s}} \norm{\delta_{T_{j-1}}}_1,$$
where  vector $\norm{\cdot}_{\infty}$ represents the largest entry in
absolute value in the vector, and hence
\ben
\nonumber
\sum_{k \geq 1} \twonorm{\delta_{T_{k}}} & \leq &
s^{-1/2} (\norm{\delta_{T_0}}_1 + \norm{\delta_{T_1}}_1 + 
\norm{\delta_{T_2}}_1 + \ldots) \\
 \label{eq::nowhere-bound-2}
& \leq & 
s^{-1/2} (\norm{\delta_{T_0}}_1 + \norm{\delta_{T_0^c}}_1) 
\; =  \;
s^{-1/2}  \norm{\delta}_1 
\\
\label{eq::nowhere-bound}
& \leq & s^{-1/2} (k_0 + 1) \norm{\delta_{T_0}}_1
\leq (k_0 + 1) \twonorm{\delta_{T_0}},
\een
where for~\eqref{eq::nowhere-bound}, we have used the fact
that for all $\delta \in E_s$
\ben
\label{eq::free-cone-max}
\norm{\delta_{T_0^c}}_1 & \leq & k_0 \norm{\delta_{T_0}}_1
\een
holds. Indeed, for $\delta$ such that~\eqref{eq::free-cone-max} holds, 
we have by~\eqref{eq::nowhere-bound-2} and~\eqref{eq::nowhere-bound}
\ben
\label{eq::cone-two-norm-2}
\twonorm{\delta} \; \leq \; 
\twonorm{\delta_{T_0}} + 
\sum_{j \geq 1} \twonorm{\delta_{T_j}}  
& \leq &
\twonorm{\delta_{T_0}} + s^{-1/2}  \norm{\delta}_1   \\
\label{eq::cone-two-norm}
& \leq & (k_0 + 2) \twonorm{\delta_{T_0}}.
\een

\subsection{On the equivalence of two RE conditions}
\label{sec:two-RE}
To introduce the second RE assumption by~\cite{BRT08}, 
we need some more notation.
For an integer $s$ such that $1 \leq s \leq p/2$, a vector 
$\upsilon \in \R^p$ and a set of indices $J_0 \subseteq \{1, \ldots, p\}$
with $|J_0| \leq s$, denoted by $J_1$ the subset of $\{1, \ldots, p\}$
corresponding to the $s$ largest in absolute value coordinates of $\upsilon$
outside of $J_0$ and defined $J_{01} \stackrel{\triangle}{=} J_0 \cup J_1$.
\begin{assumption}\textnormal{{\bf Restricted eigenvalue assumption} 
$RE(s, s, k_0, X)$~\citep{BRT08}}. 
\label{def:BRT-cond-two}
Consider a fixed design. For some integer $1\leq s \leq p/2$, and
a positive number $k_0$, the following condition holds:
\beq
\label{eq::admissible-two}
\inv{K(s, s, k_0, X)} := 
\min_{\stackrel{J_0 \subseteq \{1, \ldots, p\},}{|J_0| \leq s}}
\min_{\stackrel{\upsilon \not=0,}
{\norm{\upsilon_{J_0^c}}_1 \leq k_0 \norm{\upsilon_{J_0}}_1}}
\; \;  \frac{\norm{X \upsilon}_2}{\sqrt{n}\norm{\upsilon_{J_{01}}}_2} > 0.
\eeq
\end{assumption}

\begin{assumption}
\textnormal{\bf{Restricted eigenvalue assumption} $RE(s, s, k_0, \Sigma)$}
\label{def:BRT-cond-random}
For some integer $1\leq s \leq p/2$, and
a positive number $k_0$, the following condition holds:
\beq
\label{eq::admissible-random-duo}
\inv{K(s, s, k_0, \Sigma)} := \min_{\stackrel{J_0 \subseteq \{1, \ldots,
    p\},}{|J_0| \leq s}} 
\min_{\stackrel{\upsilon \not=0,}
{\norm{\upsilon_{J_0^c}}_1 \leq k_0 \norm{\upsilon_{J_0}}_1}}
\; \;  \frac{\norm{\Sigma^{1/2} \upsilon}_2}{\norm{\upsilon_{J_{01}}}_2} > 0.
\eeq
\end{assumption}
\begin{proposition}
\label{prop:two-RE-same}
For some integer $1\leq s \leq p/2$, and for the same $k_0 >0$, 
the two sets of RE conditions are equivalent up to a constant 
$\sqrt{2}$ factor of each other:
$$\frac{K(s, s, k_0, \Sigma)}{\sqrt{2}} \leq K(s, k_0, \Sigma) 
\leq K(s, s, k_0, \Sigma);$$
Similarly, we have
$$\frac{K(s, s, k_0, X)}{\sqrt{2}} \leq K(s, k_0, X) 
\leq K(s, s, k_0, X).$$
\end{proposition}
\begin{proof}
It is obvious that for the same $k_0 >0$, 
~\eqref{eq::admissible-random-duo} implies that the condition as 
in Definition~\ref{def:memory} holds with 
$$K(s, k_0, \Sigma) \leq K(s, s, k_0, \Sigma).$$
Now, for the other direction, suppose that  $RE(s, k_0, \Sigma)$ holds
for  for $K(s, k_0, \Sigma) >0$. Then for 
all $\upsilon \not = 0$ that is 
admissible to~\eqref{eq::admissible-random}, we have by 
Proposition~\ref{prop:magic-cone},
\ben
\twonorm{\Sigma^{1/2} \upsilon} & \geq & 
\frac{\norm{\upsilon_{T_0}}_2}{K(s, k_0, \Sigma)} \; > \; 0,
\een
where $T_0$ corresponds to locations of the $s$ largest coefficients of 
$\upsilon$ in absolute values; Now for any
$J_0 \subseteq \{1, \ldots, p\}$ such that $|J_0| \leq s$,
and $\norm{\upsilon_{J_0^c}}_1 \leq k_0 \norm{\upsilon_{J_0}}_1$ holds,
we have by\eqref{eq::admissible-random},
\ben
\label{eq::prop-continue}
\twonorm{\Sigma^{1/2} \upsilon} & \geq & 
\frac{\norm{\upsilon_{J_0}}_2}{K(s, k_0, \Sigma)} \; > \; 0.
\een
Now it is clear that $J_1 \subset T_0 \cup T_1$, and we have
for all $\upsilon \not = 0$ that is 
admissible to~\eqref{eq::admissible-random},
\ben
0\; < \; \twonorm{\upsilon_{J_{01}}}^2
& = & \twonorm{\upsilon_{J_{0}}}^2 + 
\twonorm{\upsilon_{J_{1}}}^2 \\ 
& \leq & \twonorm{\upsilon_{J_{0}}}^2 + 
\twonorm{\upsilon_{T_{0}}}^2 \\ 
& \leq & 
2 K^2(s, k_0, \Sigma) \twonorm{\Sigma^{1/2} \upsilon}^2,
\een
which immediately implies that for all $\upsilon \not = 0$ that is 
admissible to~\eqref{eq::admissible-random},
$$\frac{\twonorm{\Sigma^{1/2} \upsilon}}{
 \twonorm{\upsilon_{J_{01}}}} \geq \inv{\sqrt{2} K(s, k_0, \Sigma)} > 0.$$
Thus we have that $RE(s, s, k_0, \Sigma)$ condition holds with
$K(s, s, k_0, \Sigma) \leq \sqrt{2} K(s, k_0, \Sigma).$
The other set of inequalities follow exactly the same line of arguments.
\end{proof}

We now introduce the last assumption, for which we need some more notation.
For integers $s, m$ such that $1 \leq s \leq p/2$ and $m \geq s, s + m \leq p$,
a vector $\delta \in \R^p$ and a set of indices 
$J_0 \subseteq \{1, \ldots, p\}$ with $|J_0| \leq s$, 
denoted by $J_m$ the subset of $\{1, \ldots, p\}$ corresponding to the $m$ 
largest in absolute value coordinates of $\delta$ outside of $J_0$ and 
defined $J_{0m} \stackrel{\triangle}{=} J_0 \cup J_m$.
\begin{assumption}\textnormal{{\bf Restricted eigenvalue assumption} $RE(s,
    m, k_0, X)$~\citep{BRT08}}. 
\label{def:BRT-cond-final}
For some integer $1\leq s \leq p/2$, $m \geq s, s
+ m \leq p$, and 
a positive number $k_0$, the following condition holds:
\beq
\label{eq::admissible-final}
\inv{K(s, m, k_0, X)} := 
\min_{\stackrel{J_0 \subseteq \{1, \ldots, p\},}{|J_0| \leq s}}
\min_{\stackrel{\upsilon \not=0,}
{\norm{\upsilon_{J_0^c}}_1 \leq k_0 \norm{\upsilon_{J_0}}_1}}
\; \;  \frac{\norm{X \upsilon}_2}{\sqrt{n}\norm{\upsilon_{J_{0m}}}_2} > 0.
\eeq
\end{assumption}

\begin{proposition}
\label{prop:two-RE-same-final}
For some integer $1\leq s \leq p/2$,
$m \geq s, s + m \leq p$, and some positive number $k_0$, we have
$$\frac{K(s, m, k_0, X)}{\sqrt{2+ k_0^2}}
\leq K(s, k_0, X) \leq K(s, m, k_0, X).$$
\end{proposition}

\begin{proof}
It is clear that $K(s, k_0, X) \leq K(s, m, k_0, X)$ for $m \geq s$.
Now suppose that  $RE(s, k_0, X)$ holds, we continue 
from~\eqref{eq::prop-continue}. 
We devide $J_m$ into $J_1, J_2, \ldots$, such that 
such that $J_1$ corresponds to locations of the $s$ largest coefficients of 
$\upsilon_{J_0^c}$ in absolute values, $J_2$ corresponds to locations of the 
next $s$ largest coefficients of $\upsilon_{J_0^c}$ in absolute values, 
and so on. We first bound 
$\twonorm{\upsilon_{J_{01}^c}}^2$, following essentially the same argument
as in~\cite{CT07}:
observe that the $k$th largest value of $\upsilon_{J_0^c}$ obeys
$$\size{\upsilon_{J_{0}^c}}_{(k)} \leq  \norm{\upsilon_{J_0^c}}_1 / k;$$
Thus we have for $\delta$ that is admissible to~\eqref{eq::admissible-final},
\bens
\twonorm{\upsilon_{J_{01}^c}}^2 
& \leq & \norm{\upsilon_{J_0^c}}_1^2  \sum_{j \geq s +1} 1/k^2 
\leq s^{-1} \norm{\upsilon_{J_0^c}}_1^2  \\
& \leq & 
s^{-1} k_0^2 \norm{\upsilon_{J_0}}_1^2 
\leq k_0^2 \twonorm{\upsilon_{J_{0}}}^2.
\eens
It is clear that $\twonorm{J_1} \leq  \twonorm{T_0}$, and
\bens
0\; < \; \twonorm{\upsilon_{J_{01}}}^2
 \leq \twonorm{\upsilon_{J_{0m}}}^2 
& \leq & \twonorm{\upsilon_{J_{0}}}^2 + 
\twonorm{\upsilon_{J_{1}}}^2 + 
\twonorm{\upsilon_{J_{01}^c}}^2 \\ 
& \leq & \twonorm{\upsilon_{J_{0}}}^2 + 
\twonorm{\upsilon_{J_{1}}}^2 + 
k_0^2 \twonorm{\upsilon_{J_{0}}}^2 \\
& \leq &
(1+ k_0^2) \twonorm{\upsilon_{J_{0}}}^2 + 
\twonorm{\upsilon_{T_{0}}}^2 \\ 
& \leq & 
(2 + k_0^2) K^2(s, k_0, X) \twonorm{X \upsilon}^2,
\eens
which immediately implies that for all $\upsilon \not = 0$ that is 
admissible to~\eqref{eq::admissible-final},
$$\frac{\twonorm{X \upsilon}}{
 \twonorm{\upsilon_{J_{0m}}}} \geq \inv{\sqrt{2 + k_0^2}
K(s, k_0, X)} > 0.$$
Thus we have that $RE(s, m, k_0, X)$ condition holds with
$K(s, m, k_0, X) \leq \sqrt{2 + k_0^2} K(s, k_0, X).$
\end{proof}

\section{Results on the complexity measures}
In this section, in preparation for proving 
Lemma~\ref{lemma:U-number} and Lemma~\ref{lemma:column-space},
we first state some well-known definitions and some preliminary results 
on certain complexity measures on a set $V$ (See~\cite{MPT08} for example); 
we also provide a new result in Lemma~\ref{lemma:tilde-ell-star}.
\begin{definition}
Given a subset $U \subset \R^p$ and a number $\ve > 0$,
an $\ve$-net $\Pi$ of $U$ with respect to the Euclidean metric 
is a subset of points of $U$ such that $\ve$-balls 
centered at $\Pi$ covers $U$:
$$U \subset \bigcup_{x \in \Pi} (x + \ve \Ball_2^p),$$
where $A + B := \{a + b: a \in A, b \in B \}$ is the Minkowski sum of 
the sets $A$ and $B$.
The covering number $\Net(U, \ve)$ is the smallest cardinality of an
$\ve$-net of $U$. 
\end{definition}
Now it is well-known that there exists an absolute constant $c_1>0$ such that 
for every finite subset $\Pi \subset \Ball_2^p$, 
\ben
\label{eq::ell*}
\ell_*(\conv \Pi) = \ell_*(\Pi) \leq  c_1 \sqrt{\log|\Pi|}.
\een
The main goal of the rest of this section is to provide a bound on
a variation of the complexity measure $\ell_*(V)$, 
which we will denote with $\tilde{\ell}_*(V)$ throughout this paper,
by essentially exploiting a bound similar to~\eqref{eq::ell*} 
(cf. Lemma~\ref{lemma:tilde-ell-star}).

Given a set $V \subset \R^p$, we need to also measure $\ell_*(W)$, 
where $W$ is the subspace of $\R^p$ such that the linear function 
$\Sigma^{1/2}: V \to \R^p$ carries $t \in V$ onto:
$$W := \Sigma^{1/2} (V) = \{w \in \R^p: 
w = \Sigma^{1/2} t \text{ for some } t \in V\}.$$
We denote this new measure with $\tilde{\ell}_*(V)$.
Formally,
\begin{definition}
\label{def::entropy-sigma}
For a subset $V \subset \R^p$, we define
\beq
\label{eq::entropy-sigma}
\tilde{\ell}_*(V) := {\ell}_*(\Sigma^{1/2}(V)) := 
\expct{\sup_{t \in V} \abs{\ip{t, \Sigma^{1/2} h}}}
:= \expct{\sup_{t \in V} \abs{\sum_{i=1}^p g_i t_i}}
\eeq
where $t = (t_i)^p_{i=1} \in \R^p$, and $h = (h_i)^p_{i=1} \in \R^p$ is a 
random vector with independent $N(0, 1)$ random variables while
$g = \Sigma^{1/2} h$ is a random vector with dependent Gaussian random
variables. 
\end{definition}

We prove a bound on this measure in Lemma~\ref{lemma:tilde-ell-star} 
after we present some existing results.
The subsets that we would like to apply~\eqref{eq::entropy} 
and~\eqref{eq::entropy-sigma} are the sets consisting of 
sparse vectors: 
let $S^{p-1}$ be the unit sphere in $\R^p$, for $1 \leq m \leq p$ 
\beq
U_m \; := \; \{ x \in S^{p-1}: |\supp(x)| \leq m \}
\eeq
We shall also consider the analogous subset of the Euclidean ball,
\beq
\tilde{U}_m \; := \; \{ x \in B^{p}_2: |\supp(x)| \leq m \}
\eeq
The sets $U_m$ and $\tilde{U}_m$ are unions of the unit spheres,
and unit balls, respectively, supported on $m$-dimensional coordinate
subspaces of $\R^p$. 
The following three lemmas are well-known and mostly standard; 
See~\cite{MPT08} and~\cite{LT91} for example.

\begin{lemma}\textnormal{(\citet[Lemma2.2]{MPT08})}
\label{eq::Pi-cover-numbers}
Given $m \geq 1$ and $\ve >0$. There exists an $\ve$ cover 
$\Pi \subset B_2^m$ of $B_2^m$ with respect to the Euclidean metric such
that $B_2^m \subset (1- \ve)^{-1} \conv \Pi$ and  $|\Pi| \leq (1+2/\ve)^m$.
Similarly, there exists an $\ve$ cover of the sphere 
$S^{m-1}$, $\Pi' \subset S^{m-1}$ such that $|\Pi'| \leq (1+2/\ve)^m$.
\end{lemma}

\begin{lemma}\textnormal{(\citet[Lemma 2.3]{MPT08})} 
\label{lemma:Pi-cover}
For every $0 < \ve \leq 1/2$ and every $1\leq m \leq p$, there is a 
set $\Pi \subset \Ball_2^p$ which is an $\ve$ cover of $\tilde{U}_m$,
such that 
\ben
\label{eq::cover-number} 
\tilde{U}_m  & \subset & 2 \conv \Pi, \; \; \text{ where } \; 
|\Pi| \; \leq \; \left(\frac{5}{2\ve}\right)^m {p \choose m}
\een
Moreover, there exists an $\ve$ cover $\Pi' \subset S^{p-1}$ of 
$U_m$ with cardinality at most $\left(\frac{5}{2\ve}\right)^m {p \choose m}$.
\end{lemma}

\begin{proof}
Consider all subsets $T \subset \{1, \ldots, p\}$ with $|T| = m$,
it is clear that the required sets in $\Pi$ and $\Pi'$ in 
Lemma~\ref{lemma:Pi-cover} can be obtained
by unions of corresponding sets supported on the coordinates from $T$.
By Lemma~\ref{eq::Pi-cover-numbers}, the cardinalities of these sets
are at most $(5/2\ve)^m {p \choose m}$.
\end{proof}

\begin{lemma}\textnormal{\citep{LT91}}
\label{lemma:guass-maximality}
Let $X = (X_1, \ldots, X_N)$ be Gaussian in $\R^p$. Then 
$$\expct{\max_{i =1, \ldots, N} |X_i|} \leq 3 \sqrt{\log N} 
\max_{i =1, \ldots, N} \sqrt{\expct{X_i^2}}.$$
\end{lemma}
We now prove the key lemma that we need for 
Lemma~\ref{lemma:lower-bound-exp-imp}. The main point of the proof
follows the idea from~\cite{MPT08}: if $U_m \subset 2 \conv  \Pi_m$
for $ \Pi_m \subset B_2^p$ and there is a reasonable control of the 
cardinality of $\Pi_m$ and $\rho_{\max}(m)$ on $\Sigma$, 
then $\tilde{\ell_*}(V)$ is bounded from above.

\begin{lemma}
\label{lemma:tilde-ell-star}
Let $\Pi_m$ be a $1/2$-cover of $\tilde{U}_m$ provided by 
Lemma~\ref{lemma:Pi-cover}.
Then for $1 \leq m < p/2$ and $c = 5e$, it holds that for $V = U_m$
\ben
\label{eq::e1}
\tilde{\ell}_*(U_m) & \leq & \tilde{\ell}_*(2 \conv \Pi_m)  
= 2 \tilde{\ell}_*(\Pi_m) \; \text{ where } \\
\label{eq::e2}
\tilde{\ell}_*(\Pi_m) & \leq &
3 \sqrt{m \log c(p/m) } \sqrt{\rho_{\max}(m)}.
\een
\end{lemma}
\begin{proof}
The first inequality follows from the definition of $\tilde{\ell}_*$ and 
the fact that
\bens
V = U_m & \subset & \tilde{U}_m \; \subset \; 2 \conv \Pi_m.
\eens
The second equality in~\eqref{eq::e1} holds due to convexity which 
guarantees that 
\bens
\sup_{y \in \conv \Pi_m} \abs{\ip{y, \Sigma^{1/2} h}} & = &
\sup_{y \in \Pi_m} \abs{\ip{y, \Sigma^{1/2} h}} \text{ and hence } \\
\tilde{\ell}_*(2 \conv \Pi_m) \;  =  \; 2 \tilde{\ell}_*(\conv \Pi_m) 
& = & 2 \tilde{\ell}_*(\Pi_m).
\eens
Thus we have for $c = 5e$
\bens
\label{eq::tilde-ell*}
\tilde{\ell}_*(\conv \Pi_m)
 = \tilde{\ell}_*(\Pi_m) & := & \expct{\sup_{t \in \Pi_m} \abs{\ip{t, \Sigma^{1/2} h}}} \\
& \leq & 3 \sqrt{\log |\Pi_m|} 
\sup_{t \in \Pi_m} 
\sqrt{\expct{\abs{\ip{t, \Sigma^{1/2} h}}^2}} \\
& \leq &  
 3 \sqrt{m \log (5ep/m) }
\sup_{t \in \Pi_m} 
\twonorm{\Sigma^{1/2} t} \\
& \leq & 
3 \sqrt{m  \log c (p/m) } \sqrt{\rho_{\max}(m)}
\eens
where we have used Lemma~\ref{lemma:guass-maximality},
~\eqref{eq::cover-number} and the 
bound ${p \choose{m} } \leq \left(\frac{e p}{m} \right)^{m}$,
which is valid for $m < p/2$, and the fact that
$\expct{\abs{\ip{t, \Sigma^{1/2} h}}^2} =
\expct{\abs{\ip{h, \Sigma^{1/2} t}}^2} = \twonorm{\Sigma^{1/2} t}^2.$
\end{proof}

\subsection{Proof of Lemma~\ref{lemma:U-number}}
\begin{proofof2}
It is clear that for all $y \in \Upsilon$,  
$y = \Sigma^{1/2} \delta$ for some $\delta \in E_s$, hence
all equalities in \eqref{eq::Sigma-delta-entropy} hold.
We hence focus on bounding the last term.
For each $\delta \in E_s$, we decompose $\delta$ into a set of vectors
$\delta_{T_0}$, $\delta_{T_1}$,  $\delta_{T_2}$, \ldots, $\delta_{T_K}$ as
in Section~\ref{sec:decomp-vec}.

By Proposition~\ref{prop:magic-cone}, we have
$\norm{\delta_{T_0^c}}_1 \leq k_0 \norm{\delta_{T_0}}_1$.

\silent{
We then continue by dividing $T_0^c$ into subsets of size $s$ and enumerate 
$T_0^c$ such that $T_0^c = \bigcup_{k=1}^K T_k$, where $K \geq 1, 
|T_k| = s, \forall k = 1, \ldots, K-1$, and $|T_K| \leq s$, such that $T_1$ 
contains the indices of the largest $s$ coefficients of $\delta_{T_0^c}$ 
in absolute values, $T_2$ contains the indices of the next 
$s$ largest coefficients, and so on.} 

For each index set $T \subset \{1, \ldots, p \}$, we 
let $\delta_{T}$ represent its $0$-extended version $\delta'$ 
in $\R^p$, such that $\delta'_{T^c} = 0$ and $\delta'_{T} = \delta_T$.
For $\delta_T = 0$, it is understood that 
$\frac{\delta_T}{\twonorm{\delta_T}} :=0$ below.
Thus we have for all $\delta$ in $E_s$ and all $h \in \R^p$,
\begin{eqnarray}
\nonumber
\abs{\ip{h,  \Sigma^{1/2} \delta}}
&  = &
\abs{\ip{h,  \Sigma^{1/2} \delta_{T_{0}}
+ 
\sum_{k \geq 1} \ip{h,  \Sigma^{1/2} \delta_{T_{k}}}}} \\
\nonumber
& \leq & 
\abs{\ip{ h, \Sigma^{1/2} \delta_{T_{0}}}}
+ 
\sum_{k \geq 1} \abs{\ip{h, \Sigma^{1/2} \delta_{T_{k}}}}\\
\nonumber
& \leq & 
\abs{\ip{ \delta_{T_{0}},  \Sigma^{1/2} h}}
+ 
\sum_{k \geq 1} \twonorm{\delta_{T_{k}}} \abs{\ip{
\frac{\delta_{T_{k}}}{\twonorm{\delta_{T_{k}}}}, \Sigma^{1/2} h}} \\
\nonumber
& \leq &
\twonorm{\delta_{T_{0}}} \abs{\ip{\frac{\delta_{T_{0}}}{\twonorm{\delta_{T_{0}}}}, \Sigma^{1/2} h}}
+ 
\sum_{k \geq 1} \twonorm{\delta_{T_{k}}} 
\sup_{t \in U_{s}} \abs{\ip{t, \Sigma^{1/2} h}} \\
\nonumber
& \leq &
\left(\twonorm{\delta_{T_{0}}} + \sum_{k \geq 1} \twonorm{\delta_{T_{k}}} \right)
\sup_{t \in U_{s}} \abs{\ip{t, \Sigma^{1/2} h}}  \\
& \leq &
\label{eq::last-step}
(k_0 + 2) K(s, k_0, \Sigma) \sup_{t \in U_{s}} \abs{\ip{h, \Sigma^{1/2} t}},
\end{eqnarray}
where we have used the following bounds in~\eqref{eq::KS1} and~\eqref{eq::KS2}:
By Assumption~\ref{def:memory} and by construction of its 
corresponding sets $T_0, T_1, \ldots$, we have for all $\delta \in E_s$,
\ben
\label{eq::KS1}
\twonorm{\delta_{T_{0}}}
& \leq & K(s, k_0, \Sigma) \twonorm{\Sigma^{1/2} \delta}  
\; = \; K(s, k_0, \Sigma) \\
\label{eq::KS2}
\sum_{k \geq 1} \twonorm{\delta_{T_{k}}} 
& \leq & (k_0 +1) \twonorm{\delta_{T_0}} \; \leq \; (k_0 +1) K(s, k_0, \Sigma),
\een
where we used the bound in~\eqref{eq::nowhere-bound}.
Thus we have
by~\eqref{eq::last-step} and Lemma~\ref{lemma:tilde-ell-star}
\begin{eqnarray*}
\expct{\sup_{\delta \in E_s} \abs{\ip{h,  \Sigma^{1/2} \delta}}}
& \leq & 
(2 + k_0) K(s, k_0, \Sigma) 
\expct{\sup_{t \in U_{s}} \abs{\ip{h, \Sigma^{1/2} t}}} \\
& \leq &
(2 + k_0) K(s, k_0, \Sigma) \tilde{\ell}_*(U_{s}) \\
& \leq &
3 (2 + k_0) K(s, k_0, \Sigma)  \sqrt{s \log (cp/s)} \sqrt{\rho_{\max}(s)} \\
& : = & \bar{C} \sqrt{s \log (cp/s)}
\end{eqnarray*}
by Lemma~\ref{lemma:tilde-ell-star}, where $\bar{C}$
is as defined in~\eqref{eq::C-define} and $c = 5e$.
This proves Lemma~\ref{lemma:U-number}.
\end{proofof2}

\subsection{Proof of Lemma~\ref{lemma:column-space}}
\begin{proofof2}
Let $h_1, \ldots, h_p$ be independent $N(0, 1)$ Gaussian random variables.
We have  by Lemma~\eqref{lemma:guass-maximality},
\begin{eqnarray*}
\ell_*(\Phi) & := &
\expct{\max_{i =1, \ldots, p} \abs{\sum_{j = 1}^p \rho_{ij} h_j}} \leq
3 \sqrt{\log p} \max_{i = 1, \ldots, p}
\sqrt{\expct{\left(\sum_{j = 1}^p \rho_{ij} h_j \right)^2}} \\
& = &
3 \sqrt{\log p} \max_{i = 1, \ldots, p}
\sqrt{\sum_{j = 1}^p (\rho_{ij})^2 \expct{h_j^2}} \\
& = &
3 \sqrt{\log p} \max_{i = 1, \ldots, p} \sqrt{\Sigma_{ii}}
= 3 \sqrt{\log p},
\end{eqnarray*}
where we used the fact that $\Sigma_{ii} = 1$ for all $i$ and 
$\sigma(h_j) = 1, \forall j$.
\end{proofof2}

\section{Proofs for Theorems in Section~\ref{sec:estimators}}
\label{sec:recovery-proof}
Throughout this section, let $1> \theta > 0$.
Proving both Theorem~\ref{thm:lasso} and Theorem~\ref{thm:DS} involves 
first showing that the optimal solutions to both the Lasso and the
Dantzig selector satisfy the cone constraint as in~\eqref{eq::cone} 
for $I = \supp{\beta}$, for some $k_0 > 0$. Indeed, it holds that 
$k_0 = 1$ for the Dantzig selector  when 
$\lambda_n \geq (1+ \theta) \basepen$, 
and $k_0 =3$ for the Lasso when $\lambda_n \geq 2 (1+ \theta) \basepen$ 
(cf. Lemma~\ref{lemma:magic-number} and~\eqref{eq::DS-magic-number}).
These have been shown before, for example, in~\cite{BRT08} 
and in~\cite{CT07}. We included proofs for
(Lemma~\ref{lemma:magic-number} and~\eqref{eq::DS-magic-number})
for completeness. We then state two propositions for the Lasso estimator 
and the Dantzig selector respectively under $\T_a$,
where $a > 0$ and  $1> \theta > 0$. We first bound the
probability on $\T_a^c$.
\subsection{Bounding $\T^c_a$}
\begin{lemma}
\label{lemma:gaussian-noise}
For fixed design $X$ with $\max_j \|X_j\|_2 \le (1 + \theta) \sqrt{n}$, 
where $0< \theta < 1$,
we have for $\T_a$ as defined in~\eqref{eq::low-noise},  where $a >0$, 
$\prob{\T^c_a} \leq (\sqrt{\pi \log p} p^a)^{-1}.$
\end{lemma}
\begin{proof}
Define random variables: $Y_j = \inv{n} \sum_{i=1}^n \e_i X_{i,j}.$
Note that $\max_{1 \le j \le p} |Y_j| = \|X^T \epsilon/n\|_{\infty}$. 
We have $\expct{(Y_j)} = 0$ and
$\var{(Y_j)} = \twonorm{X_j}^2\sigma^2/n^2 \leq (1 + \theta) \sigma^2/n$. 
Let $c_0 = 1+ \theta$.
Obviously, $Y_j$ has its tail probability dominated by that of
$Z \sim N(0, \frac{c_0^2 \sigma^2}{n})$:
\begin{eqnarray*}
\prob{|Y_j| \geq t} \leq  \prob{|Z| \geq t} \leq 
\frac{2 c_0 \sigma}{\sqrt{2 \pi n} t} \exp\left(\frac{-n t^2}{2 c_0^2
    \sigma_{\e}^2}\right).
\end{eqnarray*}
We can now apply the union bound to obtain:
\bens
\prob{\max_{1 \leq j \leq p} |Y_j| \geq t } & \leq & 
p \frac{c_0 \sigma}{\sqrt{n} t} \exp\left(\frac{-n t^2}{2 c_0^2
    \sigma^2}\right) \\
&=& 
\exp\left(-\left(\frac{n t^2}{2 c_0^2 \sigma^2} + 
\log \frac{t \sqrt{\pi n}}{\sqrt{2} c_0 \sigma} - \log p\right) \right).
\eens
By choosing $t = c_0 \sigma \sqrt{1+a}\sqrt{2 \log p/n}$, 
the right-hand side is bounded by $(\sqrt{\pi \log p} p^a)^{-1}$ for $a \geq 0$.
\end{proof}

\subsection{Proof of Theorem~\ref{thm:lasso}}
\label{sec:thm-lasso-proof}
Let $\hat{\beta}$ be an optimal solution to the Lasso as in~\eqref{eq::origin}.
$S := \supp{\beta}.$ and
$$\upsilon = \hat{\beta} - \beta.$$
We first show Lemma~\ref{lemma:magic-number}; we then apply condition 
$RE(s, k_0, X)$ on $\upsilon$ with $k_0 = 3$ under $\T_a$ 
to show Proposition~\ref{prop:initial-bound-Lasso}.
Theorem~\ref{thm:lasso} follows immediately from
Proposition~\eqref{prop:initial-bound-Lasso}.
\begin{lemma} {\textnormal~\cite{BRT08}}
\label{lemma:magic-number}
Under condition $\T_a$ as defined in~\eqref{eq::low-noise},
$\norm{\upsilon_{S^c}}_1 \leq 3\norm{\upsilon_{S}}_1$
for $\lambda_n \geq 2 (1 + \theta) \basepen$ for the Lasso.
\end{lemma}
\begin{proof}
By the optimality of $\hat{\beta}$, we have
\begin{eqnarray*}
\lambda_{n} \norm{\beta}_1 - 
\lambda_{n} \norm{\hat\beta}_1 
& \geq & 
\inv{2n} \norm{Y- X  \hat\beta}^2_2  - 
\inv{2n} \norm{Y- X \beta }^2_2 \\
& \geq & 
\inv{2n} \| X \upsilon \|^2_2 - \frac{\upsilon^T X^T \e}{n} 
\end{eqnarray*}
Hence under condition $\T_a$ as in~\eqref{eq::low-noise}, we have
for $\lambda_n \geq 2 (1+\theta) \basepen$,
\begin{eqnarray}
 \nonumber
\twonorm{X \upsilon}^2/n 
& \leq &
2 \lambda_{n} \norm{\beta}_1
-  2 \lambda_{n} \norm{\hat\beta}_1 +  
2 \norm{\frac{X^T \e}{n}}_{\infty} \norm{\upsilon}_1 \\
\label{eq::precondition}
& \leq & \lambda_{n} \left(2\norm{\beta}_1 
- 2 \norm{\hat\beta}_1 + \norm{\upsilon}_1\right),
\end{eqnarray}
where by the triangle inequality, and $\beta_{\Sc} = 0$, we have
\begin{eqnarray} 
\nonumber
\lefteqn{
0 \leq 2 \norm{\beta}_1 -  2 \norm{\hat\beta}_1 + \norm{\upsilon}_1 }
\\ \nonumber
& = &
2 \norm{\beta_S}_1 - 2 \norm{\hat\beta_{S}}_1 -
2 \norm{\upsilon_{\Sc}}_1 + \norm{\upsilon_S}_1 + \norm{\upsilon_{S^c}}_1  \\
\label{eq::magic-number-2}
& \leq &
3 \norm{\upsilon_{S}}_1 - \norm{\upsilon_{S^c}}_1.
\end{eqnarray}
Thus Lemma~\ref{lemma:magic-number} holds.
\end{proof}

We now show Proposition~\ref{prop:initial-bound-Lasso}, where 
except for the $\ell_2$-convergence rate as in~\eqref{eq::2-norm-generic},
all bounds have essentially been shown in~\cite{BRT08} (as Theorem 7.2) 
under Assumption $RE(s, 3, X)$;
The bound on $\twonorm{\upsilon}$, which as far as the author is aware of, 
is new; however, this result is indeed also implied by Theorem~7.2 
in~\cite{BRT08} given Proposition~\ref{prop:two-RE-same} as derived 
in this paper. We note that the same remark holds for 
Proposition~\ref{prop:initial-bound-DS}; see~\citet[Theorem~7.1]{BRT08}.
\begin{proposition}{\textnormal{({\bf $\ell_p$-loss for the Lasso})}} 
\label{prop:initial-bound-Lasso}
Suppose that $RE(s, 3, X)$ holds.
Let $Y = X \beta + \e$, for $\e$ being i.i.d. $N(0, \sigma^2)$ and
$\twonorm{X_j}  \leq (1+ \theta) \sqrt{n}$. 
Let $\hat{\beta}$ be an optimal solution to~\eqref{eq::origin} with 
$\lambda_{n} \geq 2 (1+\theta) \basepen$, where $a \geq 0$. 
Let $\upsilon = \hat{\beta} - \beta$.
Then on condition $\T_a$ as in~\eqref{eq::low-noise},
the following hold for $B_0 =  4 K^2(s, 3, X)$ 
\begin{eqnarray}
\label{eq::S2norm}
\norm{\upsilon_{S}}_{2} & \leq & B_0 \lambda_{n} \sqrt{s}, \\
\label{eq::1-norm-generic}
\norm{\upsilon}_1 & \leq & B_0 \lambda_{n} s, \; \text{ where } \;
\norm{\upsilon_{\Sc}}_1 \leq 3 \norm{\upsilon_S}_1 \\
\label{eq::2-norm-generic}
{and} \; \; \twonorm{\upsilon} & \leq & 2 B_0 \lambda_n \sqrt{s}.
\end{eqnarray}
\end{proposition}
\begin{proof}
\label{sec:relaxed}
The first part of this proof follows that of~\cite{BRT08}.
Now under condition $\T_a$, by~\eqref{eq::precondition}
and~\eqref{eq::magic-number-2},
\begin{eqnarray}
\nonumber
\twonorm{X \upsilon}^2/n +  \lambda_{n} \norm{\upsilon}_1
& \leq &
\lambda_{n} 
\left(3 \norm{\upsilon_{S}}_1 - \norm{\upsilon_{S^c}}_1 + \norm{\upsilon_S}_1 
+ \norm{\upsilon_{\Sc}}_1 \right)  \\
\label{eq::last}
& = &  
4 \lambda_{n} \norm{\upsilon_S}_1
\leq 4\lambda_{n} \sqrt {s}  \norm{\upsilon_S}_2 \\
\label{eq::last-2} 
& \leq &
4 \lambda_{n} \sqrt{s} K(s, 3, X) \twonorm{X \upsilon}/\sqrt{n} \\
& \leq &
\label{eq::last-3} 
4 K^2(s, 3, X) \lambda_{n}^2 s + \twonorm{X \upsilon}^2/n.
\end{eqnarray}
where~\eqref{eq::last-2} holds by definition of $RE(s, 3, X)$;
Thus we have by~\eqref{eq::last-3} that
\ben
\label{eq::last-5} 
\norm{\upsilon_S}_1 \leq \norm{\upsilon}_1 \leq 
4 K^2(s, 3, X) \lambda_{n} s,
\een
which implies that ~\eqref{eq::1-norm-generic} holds with 
$B_0 =  4 K^2(s, 3, X)$.
Now by $RE(s, 3, X)$ and~\eqref{eq::last}, we have 
\begin{eqnarray}
\label{eq::last-4}
\norm{\upsilon_S}^2_2 \leq 
K^2(s, 3, X)\twonorm{X \upsilon}^2/n 
& \leq & 
K^2(s, 3, X) 4 \lambda_{n} \sqrt {s}  \norm{\upsilon_S}_2
\end{eqnarray}
which immediately implies that~\eqref{eq::S2norm} holds.

Finally, we have by~\eqref{eq::cone-two-norm-2},~\eqref{eq::last-5}, 
~\eqref{eq::cone-max} and the $RE(s, 3, X)$ condition,
\ben
\nonumber
\twonorm{\upsilon} & \leq & 
\twonorm{\upsilon_{T_0}} + s^{-1/2}  \norm{\upsilon}_1   \\
\label{eq::univers}
& \leq & 
K(s, 3, X) \twonorm{X \upsilon}/\sqrt{n} 
+ 
4 K^2(s, 3, X) \lambda_{n} \sqrt{s}, \\
\label{eq::plug-last}
& \leq & 
K(s, 3, X) \sqrt{4 \lambda_n \norm{\upsilon_S}_1 } 
+ 4 K^2(s, 3, X) \lambda_{n} \sqrt{s}, \\ 
\label{eq::plug-last-5}
& \leq &  8  \lambda_n  K^2(s, 3, X) \sqrt{s}.
\een
where in~\eqref{eq::univers}, we crucially exploit the universality 
of the RE condition; in~\eqref{eq::plug-last}, we use the bound 
in~\eqref{eq::last}; and in~\eqref{eq::plug-last-5}, we use~\eqref{eq::last-5}.
\end{proof}

\subsection{Proof of Theorem~\ref{thm:DS}}
Let $\hat{\beta}$ be an optimal solution to the Dantzig selector as 
into~\eqref{eq::DS-func}. Let $S := \supp{\beta}.$ and
$$\upsilon = \hat{\beta} - \beta.$$
We first show Lemma~\ref{lemma:magic-number-DS}; we then 
apply condition $RE(s, k_0, X)$ to $\upsilon$ with $k_0 = 1$ under $\T_a$ 
to show Proposition~\ref{prop:initial-bound-DS}.
Theorem~\ref{thm:DS} follows from immediately from
Proposition~\eqref{prop:initial-bound-DS}.

\begin{lemma}{\textnormal(\cite{CT07})}
\label{lemma:magic-number-DS}
Under condition $\T_a$, 
$\norm{\upsilon_{S^c}}_1 \leq \norm{\upsilon_{S}}_1$
for $\lambda_n \geq (1 + \theta) \basepen$, 
where $a \geq 0$ and $0< \theta < 1$ for the Dantzig selector.
\end{lemma}

\begin{proof}
Clearly the true vector $\beta$ is feasible to~\eqref{eq::DS-func}, as
\bens
\norm{\inv{n} X^T (Y - X \beta)}_{\infty} 
= \norm{\inv{n} X^T \epsilon}_{\infty}  \; \leq \; 
(1+\theta) \basepen \leq \lambda_n,
\eens
hence by the optimality of $\hat{\beta}$,
$$\norm{\hat{\beta}}_1 \leq \norm{\beta}_1.$$
Hence it holds under for $\upsilon = \hat\beta - \beta$ that
\ben
\label{eq::DS-magic-number}
\norm{\beta}_1 - \norm{\upsilon_S}_1 + \norm{\upsilon_{\Sc}}_1
\leq \norm{\beta + \upsilon} = \norm{\hat\beta}_1 \leq \norm{\beta}_1
\een
and hence $\upsilon$ obeys the cone constraint as desired.
\end{proof}

\begin{proposition}
{\textnormal{({\bf $\ell_p$-loss for the Dantzig selector})}}
\label{prop:initial-bound-DS}
Suppose that $RE(s, 1, X)$ holds. 
Let $Y = X \beta + \e$, for $\e$ being i.i.d. $N(0, \sigma^2)$ and
$\twonorm{X_j} \leq (1+\theta)  \sqrt{n}$. 
Let $\hat\beta$ be an optimal solution 
to~\eqref{eq::DS-func} with 
$\lambda_{n} \geq (1+\theta)  \basepen$, where $a \geq 0$
and $0< \theta < 1$.
Then on condition $\T_a$ as in~\eqref{eq::low-noise},
the following hold with $B_1 = 4 K^2(s, 1, X)$
\begin{eqnarray}
\label{eq::S2norm-DS}
\norm{\upsilon_{S}}_{2} & \leq & B_1 \lambda_n \sqrt{s}, \\
\label{eq::1-norm-generic-DS}
\norm{\upsilon}_1 & \leq & 2B_1 \lambda_n s, \; \text{ where } \; 
\norm{\upsilon_{S^c}}_1 \leq \norm{\upsilon_{S}}_1 \\
\text{ and } \; \; \twonorm{\upsilon} & \leq & 3 B_1 \lambda_n \sqrt{s}.
\end{eqnarray}
\end{proposition}
\begin{remark}
See comments in front of Proposition~\ref{prop:initial-bound-Lasso}.
\end{remark}
\begin{proofof}{\textnormal{Proposition~\ref{prop:initial-bound-DS}}}
Our proof follows that of~\cite{BRT08}.
Let $\hat\beta$ as an optimal solution to~\eqref{eq::DS-func}.
Let $\upsilon = \hat\beta - \beta$ and let $\T_a$ hold for $a> 0$ and 
$0 < \theta < 1$. By the constraint of~\eqref{eq::DS-func}, we have
\bens
\norm{\inv{n} X^T X \upsilon}_{\infty} \leq
\norm{\inv{n} X^T (Y - X \hat{\beta})}_{\infty} 
+ \norm{\inv{n} X^T \epsilon}_{\infty}  \; \leq \; 2 \lambda_n.
\eens
and hence by Lemma~\ref{lemma:magic-number-DS}, we have
\ben
\nonumber
\twonorm{X \upsilon}^2/n & = & \frac{\upsilon^T X^T X \upsilon}{n}
\leq
 \norm{ \inv{n}\upsilon^T X^T X}_{\infty} \norm{\upsilon}_1 
 \leq 2 \lambda_n \norm{\upsilon}_1 \\
\label{eq::DS-last-pre}
& \leq &
4 \lambda_n \norm{\upsilon_{S}}_1 \leq 
4 \lambda_n \sqrt{s} \norm{\upsilon_{S}}_2.
\een
We now apply condition $RE(s, k_0, X)$ on $\upsilon$ with $k_0 = 1$ 
to obtain
\begin{eqnarray}
\label{eq::DS-last}
\norm{\upsilon_S}^2_2 \leq K^2(s, 1, X)\twonorm{X \upsilon}^2/n 
& \leq & 
K^2(s, 1, X) 4 \lambda_n \sqrt{s} \norm{\upsilon_{S}}_2,
\end{eqnarray}
which immediately implies that ~\eqref{eq::S2norm-DS} holds.
Hence~\eqref{eq::1-norm-generic-DS} holds with $B_1 = 4 K^2(s, 1, X)$ 
given~\eqref{eq::DS-last} and
\ben
\label{eq::DS-last2}
\norm{\upsilon_{\Sc}}_1 \leq \norm{\upsilon_S}_1 
\leq 4 K^2(s, 1, X) \lambda_n s.
\een

Finally, we have by~\eqref{eq::cone-two-norm-2},~\eqref{eq::DS-last2}, 
~\eqref{eq::cone-max} and the $RE(s, 3, X)$ condition,
\ben
\nonumber
\twonorm{\upsilon} & \leq & 
\twonorm{\upsilon_{T_0}} + s^{-1/2}  \norm{\upsilon}_1   \\
\label{eq::univers-DS}
& \leq & 
K(s, 1, X) \twonorm{X \upsilon}/\sqrt{n} 
+ 
8 K^2(s, 1, X) \lambda_{n} \sqrt{s}, \\
\label{eq::plug-last-DS}
& \leq & 
K(s, 1, X) \sqrt{4 \lambda_n \norm{\upsilon_S}_1 } 
+ 8 K^2(s, 1, X) \lambda_{n} \sqrt{s}, \\ 
\label{eq::plug-last-5-DS}
& \leq &  12  \lambda_n  K^2(s, 1, X) \sqrt{s}.
\een
where in~\eqref{eq::univers-DS}, we crucially exploit the universality 
of the RE condition, and in~\eqref{eq::plug-last-DS}, we use the bound 
in~\eqref{eq::DS-last2} and~\eqref{eq::DS-last-pre}; and in
~\eqref{eq::plug-last-5-DS}, we use~\eqref{eq::DS-last2} again.
\end{proofof}

\section{A fundamental proof for the Gaussian random design}
\label{sec:Gaussian}
In this section, we state a theorem for the Gaussian random design,
following a more fundamental proof given by~\cite{RWY09} (cf. Proposition 1).
We apply their method and provide a tighter bound on the sample size 
that is required in order for $X$ to satisfy the RE condition, 
where $X$ is composed of independent rows with  multivariate Gaussian 
vectors drawn from $N(0, \Sigma)$ as in~\ref{eq::rand-des-gauss}.
We note that both upper and lower bounds in Theorem~\ref{thm:main}
are obtained  in a way that is quite similar to how the largest
and smallest singular values of a Gaussian random matrix are 
upper and lower bounded respectively; see for example~\cite{DS01}. 
The improvement over results in~\cite{RWY09} comes from the tighter 
bound on $\ell_*(\Upsilon)$ as developed in Lemma~\ref{lemma:U-number}.
Formally, we have the following.
\begin{theorem}
\label{thm:main}
Set $1 \leq n \leq p$ and $0< \theta < 1$. Consider a random design $X$ as 
in\eqref{eq::rand-des}, where $\Sigma$ 
satisfies~\eqref{eq::admissible-random} and~\eqref{eq::eigen-Sigma}.
Suppose $s < p/2$ and for $\bar{C}$ as in~\eqref{eq::C-define},
\ben
\label{eq::sample-size}
n > \inv{\theta^2} \left(\bar{C}\sqrt{s \log (5ep/s)} + \sqrt{2d \log p}\right)^2
\een
for $d > 0$. Then we have with probability at least $1- 4/p^d$, 
\begin{eqnarray}
\label{eq::gauss-bounds}
(1 - \theta -o(1)) \twonorm{\Sigma^{1/2} \delta} \; \leq \; 
\twonorm{X \delta}/\sqrt{n} 
& \leq & (1 + \theta) \twonorm{\Sigma^{1/2} \delta}
\end{eqnarray}
holds for all $\delta \not= 0$ that is admissible 
to~\eqref{eq::admissible-random}, that is, 
$\exists$ some $J_0 \in \{1, \ldots, p \}$ such that $|J_0| \leq s$
and $\norm{\delta_{J_0^c}}_1 \leq k_0 \norm{\delta_{J_0}}_1$,
where $k_0 > 0$.
\end{theorem}
\begin{proof}
We only provide a sketch here; see~\cite{RWY09} for details.
Using the Slepian's Lemma and its extension by~\cite{Gor85},
the following inequalities have been derived by~\cite{RWY09} 
(cf. Proof of Proposition 1 therein), 
\begin{eqnarray*}
\expct{\inf_{\delta \in E_s} \twonorm{X \delta}}
& \geq &
\expct{\twonorm{g}}
-  \expct{\sup_{\delta \in E_s} \abs{\ip{h,  \Sigma^{1/2} \delta}}}, \\
\expct{\sup_{\delta \in E_s} \twonorm{X \delta}}
& \leq &
\sqrt{n} + \expct{\sup_{\delta \in E_s} \abs{\ip{h,  \Sigma^{1/2} \delta}}},
\end{eqnarray*}
where $g$ and $h$ are random vectors with i.i.d  Gaussian $N(0, 1)$ 
elements in $\R^n$ and $\R^p$ respectively.
Now Lemma~\ref{lemma:lower-bound-exp-imp} follows immediately,
after we plug in the bound as in Lemma~\ref{lemma:U-number} on
$$\ell_*(\Upsilon) 
:= \expct{\sup_{\delta \in E_s} \abs{\ip{h,  \Sigma^{1/2} \delta}}}.$$ 
\begin{lemma}
\label{lemma:lower-bound-exp-imp}
Suppose $\Sigma$ satisfies Assumption~\ref{def:memory}. Then
for  $\bar{C}$ as in Theorem~\ref{thm:main}, we have
\ben
\expct{\inf_{\delta \in E_s} \twonorm{X \delta}}
& \geq & \sqrt{n} - o(\sqrt{n}) - \bar{C} \sqrt{s \log (5ep/s)} \\
\expct{\sup_{\delta \in E_s} \twonorm{X \delta}}
& \leq & \sqrt{n} +  \bar{C} \sqrt{s \log (5ep/s)}.
\een
\end{lemma}
We then apply the concentration of measure inequality for 
$\inf_{\delta \in E_s} \twonorm{X \delta}$, for which it is well
known that the $1$-Lipschitz condition holds for 
$\inf_{\delta \in E_s} \twonorm{X \delta} = 
\inf_{\delta \in E_s} \twonorm{A \Sigma^{1/2} \delta},$
where $A$ is a matrix with i.i.d. standard normal 
random variables in $\R^{n \times p}$.
Recall a function $f: X \rightarrow Y$ is called
$1$-Lipschitz condition if for all $x, y \in X$,
$$d_Y(f(x), f(y)) \leq d_X(x, y).$$
\begin{proposition}
\label{prop:lipt}
View Gaussian random matrix $A$ as a canonical Gaussian vector in $\R^{np}$.
Let $f(A) := \inf_{\delta \in E_s} \twonorm{A \Sigma^{1/2} \delta}$
and $f'(A) := \sup_{\delta \in E_s} \twonorm{A \Sigma^{1/2} \delta}$
be two functions of $A$ from $\R^{np}$ to $R$.
Then $f, f': \R^{np} \rightarrow R$ are $1$-Lipschitz:
\bens
|f(A) - f(B)| & \leq & \twonorm{A - B} \; \leq \; \fnorm{A - B},\\
|f'(A) - f'(B)| & \leq & \twonorm{A - B} \; \leq \; \fnorm{A - B}.
\eens
\end{proposition}
Finally we apply the concentration of measure in Gauss Space to obtain
for $t > 0$,
\begin{eqnarray}
\prob{|f(A) - \expct{f(A)}| > t} & \leq & 2 \exp(-t^2/2), \; \text{ and }  \\
\prob{|f'(A) - \expct{f'(A)}| > t} & \leq & 2 \exp(-t^2/2).
\end{eqnarray}
Now it is clear that with probability at least $1 - 4 /p^d$,  where $d > 0$,
we have for $X = A \Sigma^{1/2}$
\begin{eqnarray*}
\inf_{\delta \in E_s} \twonorm{A \Sigma^{1/2} \delta} & =: & f(A) \;  \geq \;
\expct{\inf_{\delta \in E_s} \twonorm{A \Sigma^{1/2} \delta}} - 
\sqrt{2 d \log p} \\
& \geq & 
\sqrt{n} -o(\sqrt{n}) - \bar{C} \sqrt{s \log (5ep/s)}
- \sqrt{2 d \log p},
\end{eqnarray*}
which we denote as event $\F$, and 
\begin{eqnarray*}
\sup_{\delta \in E_s} \twonorm{A \Sigma^{1/2} \delta} & =: & f'(A) \; \leq \; 
\expct{\sup_{\delta \in E_s} \twonorm{A \Sigma^{1/2} \delta}} +
\sqrt{2 d \log p} \\
& \leq & 
\sqrt{n} + \bar{C} \sqrt{s \log (5ep/s)} + \sqrt{2 d \log p},
\end{eqnarray*}
which we denote as $\F'$. 
Now it is clear that~\eqref{eq::gauss-bounds} holds on $\F \cap \F'$,
given~\eqref{eq::sample-size}.
\end{proof}